\documentclass[a4paper,oneside,11pt, reqno]{amsproc}

\usepackage{mathrsfs}
\usepackage{amsfonts,amssymb,amsmath,amsthm}
\usepackage{url}
\usepackage{enumerate}

\usepackage{graphicx}
\usepackage{bbm}
\usepackage{float}
\usepackage{pstricks}
\usepackage{amscd}
\usepackage{amsmath}
\usepackage{amsxtra}
\usepackage[T1]{fontenc}
\usepackage[utf8]{inputenc}
\usepackage{lipsum}
\usepackage{tikz}
\usetikzlibrary{arrows}
\usetikzlibrary{calc}
\usepackage{hyperref}

\theoremstyle{plain}
\newtheorem{theorem}{Theorem}[section]

\newtheorem{lemma}[theorem]{Lemma}
\newtheorem{proposition}[theorem]{Proposition}
\newtheorem{corollary}[theorem]{Corollary}

\theoremstyle{definition}
\newtheorem{definition}[theorem]{Definition}

\newtheorem{question}[theorem]{Question}
\newtheorem{example}[theorem]{Example}

\theoremstyle{remark}
\newtheorem{remark}[theorem]{Remark}

\newtheorem*{acknowledgments}{Acknowledgments}

\numberwithin{equation}{section}

\newcommand{\mf}{\mathfrak}
\newcommand*{\Ge}{\geqslant}
\newcommand*{\Le}{\leqslant}

\usepackage{etoolbox}

\makeatletter
\@namedef{subjclassname@2020}{%
  \textup{2020} Mathematics Subject Classification}
\makeatother

\usepackage{mathrsfs}
\usepackage{epsfig}
\usepackage[active]{srcltx}

\newcommand{\ncom}{\newcommand}
\ncom{\bq}{\begin{equation}}
\ncom{\eq}{\end{equation}}
\ncom{\beqn}{\begin{eqnarray*}}
\ncom{\eeqn}{\end{eqnarray*}}
\ncom{\beq}{\begin{eqnarray}}
\ncom{\eeq}{\end{eqnarray}}
\ncom{\nno}{\nonumber}
\ncom{\rar}{\rightarrow}
\ncom{\Rar}{\Rightarrow}
\ncom{\noin}{\noindent}
\ncom{\bc}{\begin{centre}}
\ncom{\ec}{\end{centre}}
\ncom{\sz}{\scriptsize}
\ncom{\rf}{\ref}
\ncom{\sgm}{\sigma}
\ncom{\Sgm}{\Sigma}
\ncom{\dt}{\delta}
\ncom{\Dt}{Delta}
\ncom{\lmd}{\lambda}
\ncom{\Lmd}{\Lambda}
\ncom{\eps}{\epsilon}
\ncom{\pcc}{\stackrel{P}{>}}
\ncom{\dist}{{\rm\,dist}}
\ncom{\sspan}{{\rm\,span}}
\ncom{\im}{{\rm Im\,}}
\ncom{\sgn}{{\rm sgn\,}}
\ncom{\ba}{\begin{array}}
\ncom{\ea}{\end{array}}
\ncom{\eop}{\hfill{{\rule{2.5mm}{2.5mm}}}}
\ncom{\eoe}{\hfill{{\rule{1.5mm}{1.5mm}}}}
\ncom{\eof}{\hfill{{\rule{1.5mm}{1.5mm}}}}
\ncom{\hone}{\mbox{\hspace{1em}}}
\ncom{\htwo}{\mbox{\hspace{2em}}}
\ncom{\hthree}{\mbox{\hspace{3em}}}
\ncom{\hfour}{\mbox{\hspace{4em}}}
\ncom{\hsev}{\mbox{\hspace{7em}}}
\ncom{\vone}{\vskip 2ex}
\ncom{\vtwo}{\vskip 4ex}
\ncom{\vonee}{\vskip 1.5ex}
\ncom{\vthree}{\vskip 6ex}
\ncom{\vfour}{\vspace*{8ex}}
\ncom{\norm}{\|\;\;\|}
\ncom{\integ}[4]{\int_{#1}^{#2}\,{#3}\,d{#4}}
\ncom{\inp}[2]{\langle{#1},\,{#2} \rangle}
\ncom{\Inp}[2]{\Langle{#1},\,{#2} \Langle}
\ncom{\vspan}[1]{{{\rm\,span}\#1 \}}}
\ncom{\dm}[1]{\displaystyle {#1}}

\begin{document}

\title[Joint complete monotonicity of rational functions]{Joint complete monotonicity of rational functions in two variables and toral $m$-isometric pairs}

\author[A. Anand, S. Chavan and R. Nailwal]{Akash Anand, Sameer Chavan and Rajkamal Nailwal}

\address{Department of Mathematics and Statistics\\
Indian Institute of Technology Kanpur, India}
   \email{akasha@iitk.ac.in}
  \email{chavan@iitk.ac.in}
    \email{rnailwal@iitk.ac.in}

\dedicatory{Dedicated to Professor Jan Stochel
on the occasion of his $70$th birthday}



\begin{abstract} 
We discuss the problem of classifying polynomials $p : \mathbb R^2_+ \rar (0, \infty)$  for which 
$\frac{1}{p}=\{\frac{1}{p(m, n)}\}_{m, n \Ge 0}$ is joint completely monotone, where $p$ is a linear polynomial in $y.$ 
We show that if 
$p(x, y)=a+b x+c y+d xy$ with $a > 0$ and $b, c, d \Ge 0,$ then 
$\frac{1}{p}$ is joint completely monotone  if and only if $a d - b c \Le 0.$  
We also present an application to the Cauchy dual subnormality problem for toral $3$-isometric weighted $2$-shifts.   
\end{abstract}

\maketitle


\section{A two-dimensional Hausdorff moment problem}

Let $\mathbb Z_+$ denote the set of nonnegative integers and 
let $\mathbb R_+$ denote the set of nonnegative real numbers. For a positive integer $n$ and a set $X,$ let $X^n$ stand for the $n$-fold Cartesian product of $X$ with itself. Let $\alpha = (\alpha_1, \ldots, \alpha_n), \beta
= (\beta_1, \ldots, \beta_n) \in \mathbb{Z}^n_+.$ Let $|\alpha|$ denote the sum $\alpha_1 + \cdots + \alpha_n$ and set $(\beta)_\alpha = \prod_{j=1}^n (\beta_j)_{\alpha_j},$ where $(\beta_j)_{0}=1,$ $(\beta_j)_{1}=\beta_j$ and 
 $$(\beta_j)_{\alpha_j} = \beta_j(\beta_j -1) \cdots (\beta_j -\alpha_j+1), \quad \alpha_j \Ge 2, ~j=1, \ldots, n.$$ 
We write
$\alpha \Le \beta$ if $\alpha_j \Le \beta_j$ for every $j=1, \ldots, n.$
For $\alpha \Le \beta$, we let $\binom{\beta}{\alpha}=\prod_{j=1}^n \binom{\beta_j}{\alpha_j}.$

For a net $\{a_\alpha\}_{\alpha \in \mathbb Z^n_+}$ and 
$j=1, \ldots, n,$ let $\triangle_j$ denote the {\it forward difference operator} given by
\beqn
\triangle_j a_\alpha = a_{\alpha + \varepsilon_j} - a_\alpha, \quad \alpha \in \mathbb Z^n_+,
\eeqn
where $\varepsilon_j$ stands for the $n$-tuple with $j$th entry equal to $1$ and $0$ elsewhere. Note that
the linear operators $\triangle_1, \ldots, \triangle_n$ are mutually commuting.
For $\alpha=(\alpha_1, \ldots, \alpha_n) \in \mathbb Z^n_+,$ let $\triangle^\alpha$ denote the operator $\prod_{j=1}^n \triangle^{\alpha_j}_j.$
For a function $f : \mathbb R^n_+ \rar (0, \infty)$ and 
$j=1, \ldots, n,$ let $\partial_jf$ be the $j$th partial derivative of $f$ whenever it exists. 
Note that 
the linear operators $\partial_1, \ldots, \partial_n$ are mutually commuting on the space of infinitely differentiable functions  from $\mathbb R^n_+$ into $(0, \infty).$ 
For $\alpha=(\alpha_1, \ldots, \alpha_n) \in \mathbb Z^n_+,$ let $\partial^\alpha$ denote the operator $\prod_{j=1}^n \partial^{\alpha_j}_j.$
Let $\mathbb R[x_1, \ldots, x_n]$ (for short, $\mathbb R[x]$) denote the ring of polynomials in the real variables $x_1, \ldots, x_n.$   
A polynomial $p \in \mathbb R[x]$ is said to be of {\it multi-degree} $\alpha=(\alpha_1, \ldots, \alpha_n) \in \mathbb Z^n_+$ if for each $j=1, \ldots, n,$ $\alpha_j$ is the largest integer for which $\partial^{\alpha_j}_j p \neq 0.$ 
For $\beta \in \mathbb Z^n_+,$ we say that {\it $p$ is of multi-degree at most $\beta$} if the multi-degree of $p$ is $\alpha$ and $\alpha \Le \beta.$ For a polynomial $p \in \mathbb R[x_1],$ let $\deg p$ denote the degree of $p.$ 
For the later purpose, we record the following fact, which is a consequence of \cite[Proposition 2.1]{JJS2020}. 
\beq \label{fact-m-iso}
\mbox{$\triangle^\beta p=0$ if $p$ is a polynomial of multi-degree $\gamma$ and $|\gamma| < |\beta|$}. 
\eeq 
For a set $X$ and a subset $\Omega$ of  $X,$ let $\mathbbm{1}_{_\Omega}$ denote the indicator function of $\Omega.$ For a point $x \in X,$ let $\delta_x$ denote the {\it Dirac delta measure} with point mass at $x.$ Recall that the $n$-th moment of the {\it multiplicative convolution} $\mu \diamond \nu$ of finite signed Borel measures $\mu$ and $\nu$ is the product of the $n$-th moments of $\mu$ and $\nu$ (see \cite[p.~944]{BD2005}): For a Borel measurable subset $\Omega$ of $[0, 1],$  
\beqn
\int_{[0, 1]} \mathbbm{1}_{_\Omega}(t)\mu \diamond \nu(dt) = \int_{[0, 1]} \int_{[0, 1]} \mathbbm{1}_{_\Omega}(xy)\mu(dx)\nu(dy).
\eeqn

We recall below some notions relevant for the present investigations$:$
\begin{enumerate}
\item 
 A net $\mf a = \{a_\alpha\}_{\alpha \in \mathbb{Z}^n_+}$ is said to be {\it joint completely monotone} if
 \beqn
 (-1)^{|\beta|} \triangle^{\beta} a_{\alpha} \Ge 0,  \quad \beta \in \mathbb Z^n_+, ~\alpha \in \mathbb Z^n_+.
 \eeqn
When $n=1,$ we simply refer to $\mf a$ as a {\it completely monotone} sequence. 
Following \cite{Athavale2001, W1941}, we say that a joint completely monotone net $\mf a$ is {\it minimal} if for every $j=1, \ldots, n$ and for every $\epsilon >0,$  
$\{a_{k \varepsilon_j} - \epsilon \mathbbm{1}_{_{\{0\}}}(k)\}_{k \in \mathbb Z_+}$ 
is not a completely monotone sequence.
\item An infinitely real differentiable function $f  : \mathbb R^n_+ \rightarrow (0,\infty)$ is said to be {\it joint completely monotone} if 
\beqn
    (-1)^{|\beta|}\big(\partial^{\beta} f\big)(x)\Ge 0, \quad \beta \in \mathbb Z^n_+, ~x \in \mathbb R^n_+.
\eeqn
When $n=1,$ we simply refer to $f$ as {\it completely monotone}. 
We say that $f$ is {\it separate completely monotone} if for every $x=(x_1, \ldots, x_n)$ in $\mathbb R^n_+$ and $j=1, \ldots, n,$ the function $$t \mapsto f(x_1, \ldots, x_{j-1}, t, x_{j+1}, \ldots, x_n)~\text{is completely monotone.}$$ 
\end{enumerate}
\begin{remark} \label{JCM-rmk}
Let  $f  : \mathbb R^n_+ \rightarrow (0,\infty)$ be joint completely monotone. Note that for $j = 1,\ldots, n$ and for fixed numbers $x_1,\ldots, x_{j-1}, x_{j+1},\ldots, x_n \in \mathbb R_+,$ 
\beqn
    (-1)^{m_j}\partial^{m_j}_j f(x_1, \ldots, x_{j-1}, y, x_{j+1}, \ldots, x_n) \Ge 0,  \quad y \in \mathbb{R}_+.
\eeqn
Thus $f(x_1, \ldots, x_{j-1}, \cdot, x_{j +1}, \ldots, x_n)$ is completely monotone, and since $j$ is arbitrary, $f$ is separate completely monotone.
A similar remark applies to the joint completely monotone nets. 
\hfill $\diamondsuit$
\end{remark}

By the solution of the multi-dimensional Hausdorff moment problem (see \cite[Proposition~4.6.11]{BCR1984}), $\mf a = \{a_\alpha\}_{\alpha \in \mathbb{Z}^n_+}$ is a joint completely monotone net if and only if it is a {\it Hausdorff moment net}, that is, if there exists a finite positive regular Borel measure $\mu$ concentrated on $[0,1]^n$ such that 
\beqn
a_\alpha =\int _{[0,1]^n} t^{\alpha} \mu(dt), \quad \alpha \in \mathbb Z^n_+.
\eeqn 
If such a measure $\mu$ exists, then it is unique. This is a consequence of the $n$-dimensional Weierstrass theorem and the Riesz representation theorem (see  \cite[Theorem~2.14]{R2006} and \cite[Lemma 4.11.3]{S2015}). 
We refer to $\mu$ as the {\it representing measure} of $\mf a.$
For more information on joint completely monotone functions and its variants, the reader is referred to \cite{BCR1984, R2014, SSV}.

The investigations in this paper are motivated by the following multi-variable analog of \cite[Question~1.8]{AC2017}$:$

\begin{question} \label{Q1} Let $n$ be a positive integer. 
For which polynomials  $p : \mathbb R^n_+ \rar (0, \infty),$ the net $\big\{\frac{1}{p(\alpha)}\big\}_{\alpha \in \mathbb Z^n_+}$ is joint completely monotone (resp. minimal joint completely monotone)? If $\big\{\frac{1}{p(\alpha)}\big\}_{\alpha \in \mathbb Z^n_+}$ is joint completely monotone, then what is the representing measure of $\big\{\frac{1}{p(\alpha)}\big\}_{\alpha \in \mathbb Z^n_+}$?
\end{question}

A motivation for this moment problem comes from the Cauchy dual subnormality problem for toral $m$-isometries (see \cite{AC2017, ACJS, AS1999, BS2019, CC2012, RZ2019}; the reader is referred to  \cite{Ag1990, AS1995, Athavale2001, AS1999, CC2012, JJS2020, S2001} for the basic theory of toral $m$-isometries).  
Before we state the main result of this paper providing a partial answer to the aforementioned question in the case of $n=2,$ we discuss a couple of instructive nonexamples$:$ 

\begin{example} \label{exam-1.3-i}
 Let $p : \mathbb R^2_+ \rightarrow (0,\infty)$ be given polynomial. 
Let us see some situations where $\frac{1}{p}$ fails to be joint completely monotone$:$
\begin{enumerate}
\item Consider $p(x,y)=a+bx+cx^2+dy,$ $x, y \in \mathbb R_+.$ A routine calculation shows that $a > 0$ and $c, d \Ge 0.$ If $d \neq 0,$ then for some $y_0 \in \mathbb R_+$, $p_{y_0}:=p(\cdot,y_0)$ has complex roots (since $b^2- 4 (a+dy)c < 0$ for large values of $y$).  Hence, by \cite[Proposition~4.3]{AC2017}, $\frac{1}{p}$ is never separate completely monotone, and hence, by Remark~\ref{JCM-rmk}, it is not joint completely monotone.
\item 
Consider $p(x, y)=1+y+xy,$ $x, y \in \mathbb R_+.$  Since the reciprocal of  any polynomial from $\mathbb R_+$ into $(0, \infty)$ of degree $1$ is completely monotone, $\frac{1}{p}$ is separate completely monotone. Note that $$\partial_1 \partial_2\Big(\frac{1}{p}\Big)(x, y)=\frac{yx+y-1}{(1+y+yx)^3}, \quad x, y \in \mathbb R_+.$$ Clearly, for $x, y \in (0, 1/2),$ $\partial_1 \partial_2(\frac{1}{p})(x, y) < 0,$ and hence $\frac{1}{p}$ is not joint completely monotone.
\end{enumerate}
Thus, the joint complete monotonicity of $\frac{1}{p}$ may fail for a polynomial of multi-degree $(1, 1).$
\eof
\end{example}

The above example raises the problem of describing all polynomials $p : \mathbb R^2_+ \rar (0, \infty)$ of bi-degree $(1, 1)$ for which the net 
$\big\{\frac{1}{p(\alpha)}\big\}_{\alpha \in \mathbb Z^2_+}$ is joint completely monotone. Needless to say, the following result leads to a complete solution to this intermediate problem (see Theorem~\ref{deg-1-1}).

\begin{theorem}[Main theorem] \label{coro-main}
For $a_j, b_j \in (0, \infty),$ $j=0, \ldots, k,$ let  
\beq 
\label{a-and-b}
a(x)= a_0 \prod_{j=1}^k (x+a_j), ~b(x)= b_0 \prod_{j=1}^k (x+b_j), \quad x \in \mathbb{R}_+.
\eeq
The following statements are valid$:$
\begin{enumerate}
\item[$\mathrm{(i)}$] 
$\big\{\frac{1}{b(m) + a(m)n}\big\}_{m,n \in \mathbb{Z}_+}$ is a minimal joint completely monotone net if  
\beq \label{sufficient-c-2}
b_1 \Le a_1 \Le b_2 \Le a_2 \Le \ldots \Le b_k \Le a_k,
\eeq
\item[$\mathrm{(ii)}$] if 
$\big\{\frac{1}{b(m) + a(m)n}\big\}_{m,n \in \mathbb{Z}_+}$ is a joint completely monotone net, then 
\beq \label{necessary-c-2}
&& \displaystyle \sum_{j=1}^k  \frac{1}{a_j}  \,  \Le \, \sum_{j=1}^k  \frac{1}{b_j}, \\  && \prod_{j=1}^k b_j  \, \Le \,  \prod_{j=1}^k a_j, 
\label{necessary-c-3}
\\
&& \sum_{j=1}^{k}b_j \, \Le \, \sum_{j=1}^{k} a_j. \label{necessary-c-1}
\eeq
\end{enumerate}
\end{theorem}
\begin{remark} 
The condition \eqref{sufficient-c-2} is not necessary for $\big\{\frac{1}{p(m, n)}\big\}_{m, n \in \mathbb Z_+}$ to be joint completely monotone.  
For example, for $$p(x, y) = (x+1)(x+4) + (x+2)(x+3)y,$$ 
the net $\big\{\frac{1}{p(m, n)}\big\}_{m, n \in \mathbb Z_+}$ is joint completely monotone. This may be concluded from Theorem~\ref{thm-bi-deg-2-1}(i) (to be seen later).
On the other hand, \eqref{necessary-c-2}, \eqref{necessary-c-3} and \eqref{necessary-c-1} together are not sufficient to guarantee the joint complete monotonicity of $\big\{\frac{1}{p(m, n)}\big\}_{m, n \in \mathbb Z_+}.$ Indeed, if $$p(x, y) = \left(x+1\right)\left(x+2\right)+\left(x+3\right)\left(x+4\right)y,$$ then 
by \cite[Proposition~4.3]{AC2017}, the sequence $\big\{\frac{1}{p(m, 1)}\big\}_{m \in \mathbb Z_+}$ is not completely monotone, and hence the net
$\big\{\frac{1}{p(m, n)}\big\}_{m, n \in \mathbb Z_+}$  is not joint completely monotone (see Remark~\ref{JCM-rmk}). Finally, note that if 
$\big\{\frac{1}{b(m) + a(m)n}\big\}_{m,n \in \mathbb{Z}_+}$ is a joint completely monotone net, then
$b_1 \Le a_k.$
Indeed, if $a_k < b_1,$ then 
\beqn
\frac{1}{b_k} \Le \ldots \Le \frac{1}{b_1}  < \frac{1}{a_k} \Le \ldots \Le \frac{1}{a_1},
\eeqn
which contradicts the inequality in \eqref{necessary-c-2}. 
\hfill $\diamondsuit$
\end{remark}

\subsection*{Plan of the paper}
A large portion of Section~2 occupies proof of Theorem~\ref{coro-main}. The proof begins with a generalization of Theorem~\ref{coro-main}(i) (see Theorem~\ref{main-thm}). We present two proofs of Theorem~\ref{main-thm}(a). The first of which exploits the partial fraction decomposition of rational functions in one variable (see \cite{AC2017, RZ2019}) together with a theorem of 
Berg and Dur\'{a}n from \cite{BD2005}. The second proof is more elementary in approach and provides a method to compute the representing measures in question.
We also present several consequences of Theorem \ref{main-thm}. The first two of which are outcomes of the proof of Theorem \ref{main-thm} (see Corollary~\ref{main-thm-coro} and Corollary~\ref{coro-proof}). 
We also generalize Theorem~\ref{coro-main}(i) providing a couple of families of joint completely monotone rationals functions having nonconstant numerator (see Corollary~\ref{coro-main-thm}). 

In Section~3, we present answers to Question~\ref{Q1} in several subcases of lower bi-degree. The first result of this section describes all polynomials $p : \mathbb R^2_+ \rar (0, \infty)$ of bi-degree at most $(1, 1)$ for which $\frac{1}{p}$ is joint completely monotone and also compute the representing measures in question (see Theorem~\ref{deg-1-1}). In the case of bi-degree $(2, 1),$ we show that an answer to Question~\ref{Q1} boils down to two subcases (see Proposition~\ref{prop-deg-2-1}). We then provide partial answers to Question~\ref{Q1} in these subcases  (see Theorems~\ref{thm-bi-deg-2-1} and \ref{bi-deg-2-1-second}). In particular, we improve Theorem~\ref{coro-main}(i) in the case of $k=2.$

In Section~4, we provide a solution to the Cauchy dual subnormality problem for toral $3$-isometric weighted $2$-shifts which are separate $2$-isometries (see Definition~\ref{def-sep-m-iso} and Theorem~\ref{CDSP-2-iso-wt}). The proof of Theorem~\ref{CDSP-2-iso-wt} relies on Theorem~\ref{deg-1-1} and characterizations of toral $m$-isometries (see Theorem~\ref{thm-char-m-iso}) and of separate $2$-isometries within the class of toral $m$-isometries (see Corollary~\ref{coro-char-m-iso}). The operator tuple torally Cauchy dual to a toral $2$-isometric weighted $2$-shift is always jointly subnormal (see Corollary~\ref{yes-C-dual}). We also exhibit a family of toral $3$-isometric weighted $2$-shift without jointly subnormal toral Cauchy dual (see Example~\ref{not-C-dual}).

We conclude the paper with a brief discussion on the role of the so-called coefficient-matrix in the moment problem addressed in  Question~\ref{Q1}. In particular, we employ the coefficient-matrix to reformulate some of the results of Section~3 (see Theorem~\ref{matrix-i} and \eqref{matrix-s-n}).

\section{Proof of the main theorem and its consequences}

The first step towards the proof of Theorem~\ref{coro-main} is the following fact$:$
\begin{theorem} \label{main-thm}
Let $a, b$ be as given in \eqref{a-and-b}
and let $\{c(m)\}_{m \in \mathbb Z_+}$ be a sequence of positive real numbers.  Assume that 
\begin{enumerate}
\item[$\mathrm{(A1)}$] $b_1 \Le a_1 \Le b_2 \Le a_2 \Le \ldots \Le b_k \Le a_k,$
\item[$\mathrm{(A2)}$] $\big\{\frac{c(m)}{a(m)}\big\}_{m \in \mathbb Z_+}$ is a completely monotone sequence. 
\end{enumerate}
 Then the following statements are valid$:$
 \begin{enumerate}
\item[$\mathrm{(a)}$] the net $\big\{\frac{c(m)}{b(m) + a(m)n}\big\}_{m,n \in \mathbb{Z}_+}$ is joint completely monotone, 
\item[$\mathrm{(b)}$] $\big\{\frac{c(m)}{b(m) + a(m)n}\big\}_{m,n \in \mathbb{Z}_+}$ is a minimal joint completely monotone net provided $\big\{\frac{c(m)}{a(m)}\big\}_{m \in \mathbb Z_+}$ is a minimal completely monotone sequence. 
\end{enumerate}
\end{theorem}

In the proof of Theorem~\ref{main-thm}, 
we need the following identity: For any real number $x >0,$ 
\beq \label{EQN2}
\frac{(-1)^{l-1}}{(l-1)!}\int_{0}^{1}(\log{s})^{l-1}s^{x-1+n}ds=\frac{1}{(n + x)^l}, \quad l \Ge 1, ~n \Ge 0
\eeq
(cf.~\cite[Eqn~(3.2)]{AC2017}). 
For the sake of completeness, we verify this identity by induction on $l \Ge 1$. Fix $x >0.$ For $l=1$, we see that
\begin{align*}
\frac{(-1)^{l-1}}{(l-1)!}\int_0^1 (\log s)^{l-1}s^{x-1+n}\,ds = \int_0^1 s^{x-1+n}\,ds = \frac{1}{n+x}, \quad n \Ge 0,
\end{align*}
and hence, the equation \eqref{EQN2} holds for $l=1$. Now assuming \eqref{EQN2} for $l \Ge 1$, we note that for any nonnegative integer $n,$
\beqn
&&\frac{(-1)^{l}}{l!}\int_0^1 (\log s)^{l}s^{x-1+n}\,ds \\
&=&\frac{(-1)^{l}}{l!} \left( - \lim_{s\to 0} \left( (\log s)^{l} \frac{s^{x+n}}{x+n} \right) - \frac{l}{n+x}\int_0^1 (\log s)^{l-1}s^{x-1+n}\,ds \right) \\ 
& =& \frac{1}{n+x} \left( \frac{(-1)^{l-1}}{(l-1)!} \int_0^1 (\log s)^{l-1}s^{x-1+n}\,ds \right) \\
&=& \frac{1}{(n+x)^{l+1}}.
\eeqn
This completes the verification of \eqref{EQN2}. 
In a proof of Theorem~\ref{main-thm}, we also need a special case of the following general fact.
\begin{lemma}\label{Sev-vari}
For $b_0 >0,$ $m_i \in \mathbb Z_+$ and $b_i \in (0, \infty),$ $i=1, \ldots, k,$ let $p : \mathbb R_+\rightarrow (0,\infty)$ be any polynomial and $q : \mathbb R_+ \rightarrow (0,\infty)$ be a non-constant polynomial given by $q(x) = b_0 \prod_{i=1}^{k}(x+b_i)^{m_{i}}.$ If $\deg q \Ge \deg{p},$ then the following statements are valid$:$
\begin{enumerate}
	\item[$\mathrm{(i)}$] there exist unique $c_0 \in \mathbb R_+$ and $c_{ij} \in \mathbb{R}$ such that $$\frac{p(x)}{q(x)}=c_0+\sum_{i=1}^{k}\sum_{j=1}^{m_i}\frac{c_{ij}}{(x+b_i)^{j}}, \quad x \in \mathbb R_+,$$
	\item[$\mathrm{(ii)}$] if $\displaystyle \sum_{i=1}^{k}\sum_{j=1}^{m_i} 
	\frac{c_{ij}}{(j-1)!}(-\log s)^{j-1}s^{b_i} \Le 0$ for every $s\in (0,1),$ then for every $t \in (0,1)$, 
	$\big\{t^{\frac{p(n)}{q(n)}}\big\}_{n \in \mathbb{Z}_+}$ is a completely monotone sequence.
\end{enumerate}
\end{lemma}
\begin{proof} 
	(i) If $\deg p < \deg q,$ then this fact is precisely \cite[Proposition~2.1]{RZ2019}. If $\deg p=\deg q,$ then by the  polynomial long division (see \cite[pg.~271, Example(2)]{DF2004}), there exists $c_0 \in \mathbb R$ and a polynomial $r$ such that $\deg r <\deg q$ and 
		\beqn
		\frac{p(x)}{q(x)}=c_0+\frac{r(x)}{q(x)}, \quad x \in \mathbb{R}_+.
		\eeqn
	Indeed,	$c_0=\alpha_p/b_0 >0,$ where $\alpha_p$ denotes the leading coefficient of the polynomial $p.$ 
Another application of \cite[Proposition~2.1]{RZ2019} (applied to $r$ and $q$) now yields the desired result. 
	 
	(ii) Assume that
	\beq \label{assumption-ii-fact}
	\sum_{i=1}^{k}\sum_{j=1}^{m_i} \frac{c_{ij}}{(j-1)!}(-\log s)^{j-1}s^{b_i} \Le 0, \quad s \in (0,1).
	\eeq
	For every $n \in \mathbb Z_+,$ by (i), we have
\allowdisplaybreaks
		\beqn 
		\frac{p(n)}{q(n)}  & = & c_0 + \sum_{i=1}^{k}\sum_{j=1}^{m_i}\frac{c_{ij}}{(n+b_i)^{j}} \\
		 & \overset{\eqref{EQN2}}= & c_0 + \sum_{i=1}^{k}\sum_{j=1}^{m_i}\int_{[0, 1]} \frac{c_{ij}}{(j-1)!}(-\log s)^{j-1}s^{b_i-1+n}ds\\
		& = & c_0  + \int_{[0, 1]} s^n \Big(\sum_{i=1}^{k}\sum_{j=1}^{m_i}\frac{c_{ij}}{(j-1)!}(-\log s)^{j-1}s^{b_i-1}\Big)ds.
		\eeqn
		This yields for every $t \in (0,1)$ and every $n \in \mathbb Z_+,$
		\beqn
	&&	(\log t)\Big(\frac{ p(n)}{q(n)}-c_0\Big) \\
	&=& \int_{[0, 1]} s^n (\log t)\Big(\sum_{i=1}^{k}\sum_{j=1}^{m_i}\frac{c_{ij}}{(j-1)!}(-\log s)^{j-1}s^{b_i-1}\Big)ds.
		\eeqn
This combined with the assumption \eqref{assumption-ii-fact} shows that for every $t \in (0,1),$ $\big\{(\log t)\big(\frac{ p(n)}{q(n)}-c_0\big)\big\}_{n \in \mathbb{Z}_+}$ is a Hausdorff moment sequence. Hence, by \cite[Corollary~2.2]{BD2005}, for every $t \in (0,1)$, 
$\big\{ e^{(\log t)\big(\frac{ p(n)}{q(n)}-c_0\big)}\big\}_{n \in \mathbb{Z}_+},$ or equivalently, $\big\{t^{\frac{p(n)}{q(n)}}\big\}_{n\in \mathbb{Z}_+}$  is a Hausdorff moment sequence.
\end{proof}

A key step in the proof of Theorem~\ref{main-thm}(a) reduces the two-dimensional moment problem in question to a continuum of one-dimensional moment problems. 

\begin{lemma} \label{continuum}
Let $a, b$ be as given in \eqref{a-and-b}
and let $\{c(m)\}_{m \in \mathbb Z_+}$ be a sequence of positive real numbers. Assume that $(A2)$ holds. 
Then the net $\big\{\frac{c(m)}{b(m) + a(m)n}\big\}_{m,n \in \mathbb{Z}_+}$ is joint completely monotone provided
\beq
\label{claim}		
\text{$\big\{t^{\frac{b(m)}{a(m)}}\big\}_{m \in \mathbb Z_+}$ is a Hausdorff moment sequence for every $t \in (0, 1).$}
\eeq
\end{lemma}
\begin{proof} Assume that \eqref{claim} holds. 
Let $A=a/c$ and $B = b/c.$ 
Note that 
		\beq \label{EQN4}
	\frac{c(m)}{b(m) + a(m)n} &=&	\frac{1}{B(m)+A(m)n} \notag \\
	&=& \int_{[0, 1]} t^n\,\frac{t^{\frac{B(m)}{A(m)}-1}}{A(m)}dt,  
	\quad m,n \in \mathbb{Z}_+.
		\eeq
	Since $\frac{B}{A}=\frac{b}{a},$	it suffices to check that for every $t \in (0, 1),$ $\Big\{\frac{t^{\frac{b(m)}{a(m)}-1}}{A(m)}\Big\}_{m \in \mathbb Z_+}$ is a Hausdorff moment sequence. In turn, in view of the assumption (A2) (which ensures that $\big\{\frac{1}{A(m)}\big\}_{m \in \mathbb Z_+}$ is a completely monotone sequence)  and the fact that the product of two completely monotone sequences is completely monotone (see  \cite[Lemma~8.2.1(v)]{BCR1984}), 
this follows from \eqref{claim}.
\end{proof}

\begin{proof}[Proof I of Theorem~\ref{main-thm}$(a)$]  
In view of Lemma~\ref{continuum}, it is sufficient to check that \eqref{claim} holds. 
In case $b_j=b_{j+1}$ for some $j = 1, \ldots, k-1,$ then by the assumption (A1), $a_j=b_j.$ In this case, after cancelling the factors $x+a_j$ and $x+b_j$ from $\frac{B}{A}=\frac{b}{a},$ we may assume that all $b_j$'s are distinct. Repeating the same argument, we may assume that $a_j$'s are also distinct. 
By Lemma~\ref{Sev-vari}(i) (applied to $p=b,$ $q=a$ and $m_i=1$ for all $i$), there exist $c_0, c_{1}, \ldots, c_k \in \mathbb{R}$ such that 
\beq \label{p-frac}	
	\frac{B(m)}{A(m)}=	\frac{b(m)}{a(m)} =  c_0 + \sum_{j=1}^{k} \frac{c_{j}}{m+a_j}, \quad m \in \mathbb Z_+.
\eeq
Clearly, $c_0=\frac{b_0}{a_0}.$ Also, by  \cite[Proposition~2.1]{RZ2019} (applied with all $b_i=1$),
\beq 
\label{formula-c-j-Ramiz}
c_j=\frac{b(-a_j)}{a_0 \prod_{1 \Le l \neq j \Le k}(a_l-a_j)}, \quad j =1, \ldots, k.
\eeq
Since $b(x)=b_0 \prod_{i=1}^{k}(x+b_i),$  
\beq \label{formula-c-j}
c_j &=&  \frac{b_0}{a_0} \frac{\prod_{l=1}^k (b_l-a_j)}{\prod_{1 \Le l \neq j \Le k}(a_l-a_j)} \\ \notag
&=& - \frac{b_0}{a_0} \frac{\prod_{1 \Le l \Le j} (a_j-b_l)}{\prod_{1 \Le l \Le j-1}(a_j-a_l)}\frac{\prod_{j+1 \Le l \Le k} (b_l-a_j)}{\prod_{j+1 \Le l \Le k}(a_l-a_j)},
\quad j=1, \ldots, k,
\eeq		
which is a non-positive real number in view of the assumption (A1). 
Thus $\sum_{j=1}^{k} c_{j} s^{a_j} \Le 0$ for every $s\in (0,1).$ Hence, we get \eqref{claim} from Lemma~\ref{Sev-vari}(ii) (applied to $p=b$ and $q=a$), completing the Proof I of Theorem~\ref{main-thm}$(a).$ 
\end{proof}
Although Lemma~\ref{Sev-vari} provides an elegant proof of Theorem~\ref{main-thm}(a), 
a close examination of Proof I shows that an application of Lemma~\ref{Sev-vari}(ii) could be replaced by an argument based on multiplicative convolution of measures
(cf. \cite[Proof of Lemma~2.1]{BD2005}). Moreover, the following alternate proof of Theorem~\ref{main-thm}(a) provides an algorithm to compute the representing measure in question.
\begin{proof}[Proof II of Theorem~\ref{main-thm}$(a)$] 
Note that by \eqref{p-frac}, 
		\beq \label{s-power-prod}
		t^{\frac{B(m)}{A(m)}-1}= t^{\frac{b_0}{a_0}-1}\, t^{ \sum_{j=1}^{k}\frac{c_j}{m+a_j}} = 
	t^{\frac{b_0}{a_0}-1}\,	\prod_{j=1}^{k}t^{\frac{c_j}{m+a_j}}, \quad t >0, ~m \in \mathbb Z_+.
		\eeq
 For any $j=1, \ldots, k,$ $m \in \mathbb Z_+$ and $t>0,$ consider
		\beq \label{bidegree-second-case}
	t^{\frac{c_j}{m+a_j}} &=& \sum_{l=0}^{\infty}\frac{(c_j \log t)^{l}}{{l!} (m+a_j)^{l}} \\
	&\overset{\eqref{EQN2}}=&
	\int_{[0, 1]} s^m \delta_1(ds)+ \sum_{l=1}^{\infty} \frac{ (c_j \log t)^l }{(l-1)!l!}  \int_{[0, 1]} (-\log {s})^{l-1}s^{a_j-1+m}ds. \notag
\eeq
By the dominated convergence theorem, we obtain
\beq \label{each-t-int}
		t^{\frac{1}{k}\big(\frac{b_0}{a_0}-1\big)} t^{\frac{c_j}{m+a_j}} 
=	\int_{[0, 1]} s^m \mu_{j, t}(ds), \quad m \in \mathbb Z_+, ~j=1. \ldots, k,
		\eeq
		where the measure $\mu_{j, t}$ is of the form
		\beq \label{mu-j-new}
		\mu_{j, t}(ds) = t^{\frac{1}{k}\big(\frac{b_0}{a_0}-1\big)} \delta_1(ds) + w_j(s, t)ds 
		\eeq
		with the weight function $w_j$ given by 
\beq
	\label{form-w-j} 
		w_j(s,t)  
	=  t^{\frac{1}{k}\big(\frac{b_0}{a_0}-1\big)} s^{a_j-1} \sum_{l=1}^{\infty}\frac{(c_j \log t)^{l}{(- \log s)}^{l-1}}{(l-1)!l!}, ~  s, t \in (0, 1). 
\eeq
Clearly, $w_j$ integrable with respect to the Lebesgue measure on $[0, 1]$ for every $j=1, \ldots, k.$ 
By the assumption (A2), there exists a positive finite Borel measure $\mu$ on $[0, 1]$ such that
\beqn 
\frac{c(m)}{a(m)} = \int_{[0, 1]} s^m \mu(ds), \quad m \in \mathbb Z_+. 
\eeqn 	
		Combining this with \eqref{s-power-prod} and \eqref{each-t-int},  we obtain  
		\beq \label{weighted-L}
		\frac{t^{\frac{B(m)}{A(m)}-1}}{A(m)} &=&  \frac{c(m)}{a(m)}
	\prod_{j=1}^{k}	t^{\frac{1}{k}\big(\frac{b_0}{a_0}-1\big)} t^{\frac{c_j}{m+a_j}}  \notag\\
&=& \int_{[0, 1]} s^m \mu(ds) \prod_{j=1}^{k} 
\int_{[0, 1]} s^m \mu_{j, t}(d(s,t)) \\ \notag 
&=& \int_{[0, 1]} s^m \nu_t(ds), \quad t \in (0, 1), 
		\eeq
		where $\nu_t$ is the multiplicative convolution of $\mu$ and $\mu_{j, t},$ $j=1, \ldots, k.$ 
This combined with \eqref{EQN4} yields
		\beq
		\label{rep-main-seq}
		\frac{c(m)}{b(m) + a(m)n}=\int_{[0, 1]} \int_{[0, 1]} s^mt^n \nu_t(ds) dt.
		\eeq
		
By \eqref{formula-c-j} and the assumption (A1),	$c_1, \ldots, c_k \Le 0.$  It now follows from \eqref{form-w-j} that $w_j$ is nonnegative on $(0, 1)$ for every $j=1, \ldots, k.$ Since the multiplicative convolution of positive measures is positive, this combined with \eqref{mu-j-new} and \eqref{rep-main-seq} completes Proof II of (a).
	\end{proof}
	\begin{remark} \label{assu-A1-used}
	The assumption (A1) is used only in the last paragraph of Proof II of Theorem~\ref{main-thm}$(a)$ to conclude that $c_1 , \ldots, c_k \Le 0.$ \hfill $\diamondsuit$
	\end{remark}

To prove Theorem~\ref{main-thm}$(b)$, we need an elementary fact pertaining to the convolution of finite signed measures. 
\begin{lemma} \label{convolution-mr}
 Let $\mu, \nu$ be finite signed Borel measures on $[0,1]$ and let $\mu \diamond \nu$ be the multiplicative convolution of $\mu$ and $\nu.$ 
 If $\mu(\{0\})=0$ and $\nu(\{0\})=0,$ then 
 $\mu \diamond \nu(\{0\})=0.$
 \end{lemma}
 \begin{proof}
Note that 
 \beqn
 \mu \diamond \nu(\{0\}) &=& \int_{[0,1]}\mathbbm{1}_{\{0\}}(x) \mu \diamond \nu(dx)\\
  &=& \int_{[0,1]\times [0,1]}\mathbbm{1}_{\{0\}}(xy) \mu(dx) \nu(dy)\\
  &=& \int_{([0, 1] \times \{0\}) \cup (\{0\} \times [0, 1])}\mathbbm{1}_{\{0\}}(xy) \mu(dx) \nu(dy)\\
   &=& \int_{[0, 1] \times \{0\}} \mu(dx) \nu(dy)+ \int_{\{0\} \times (0, 1]}\mu(dx) \nu(dy)\\
    &=&  \mu([0,1])\nu(\{0\})+ \mu(\{0\})\nu((0,1]).
 \eeqn
 This yields the desired conclusion. 
 \end{proof}

\begin{proof}[Proof of Theorem~\ref{main-thm}$(b)$]  
Assume that $\big\{\frac{c(m)}{a(m)}\big\}_{m \in \mathbb{Z}_+}$ is a minimal completely monotone sequence with the representing measure $\mu.$
In view of \cite[Proposition~5]{Athavale2001}, it suffices to check that 
\beqn
\eta([0, 1] \times \{0\})=0,~ \eta(\{0\} \times [0, 1])=0,
\eeqn
where $\eta$ is the representing measure of $\big\{\frac{c(m)}{b(m) + a(m)n}\big\}_{m,n \in \mathbb{Z}_+}.$  
Let $\nu_t$ be the multiplicative convolution of $\mu$ and $\mu_{j, t},$ $j=1, \ldots, k$ (see \eqref{mu-j-new}). 
It follows from \eqref{rep-main-seq} that 
 \beqn
 \eta([0, 1] \times \{0\}) &=& \int_{[0, 1] \times [0, 1]} \mathbbm{1}_{[0, 1] \times \{0\}}(s, t)\nu_t(ds)dt \\ &=& \int_{[0, 1]} \mathbbm{1}_{\{0\}}(t) \Big(\int_{[0, 1]} \nu_t(ds)\Big) dt,
 \eeqn
 which is clearly $0.$   
Since $\big\{\frac{c(m)}{a(m)}\big\}_{m \in \mathbb{Z}_+}$ is a minimal completely monotone sequence,  
by \cite[Theorem~IV.14a]{W1941}, $\mu(\{0\})=0.$ 
Further, since $\mu_{j, t}(\{0\})=0,$ $j=1, \ldots, k$ (see \eqref{mu-j-new}), repeated applications of Lemma~\ref{convolution-mr} ($k-1$ times) show that $\nu_t(\{0\})=0$ for any $t \in (0, 1).$
It now follows that  
 \beqn
 \eta(\{0\} \times [0, 1]) &=& \int_{[0, 1] \times [0, 1]} \mathbbm{1}_{\{0\} \times [0, 1]}(s, t)\nu_t(ds)dt \\ &=& \int_{[0, 1]} \Big(\int_{[0, 1]} \mathbbm{1}_{\{0\}}(s)\nu_t(ds)\Big) dt \\
&=& \int_{[0, 1]} \nu_t(\{0\})dt,
 \eeqn
which shows that $\eta(\{0\} \times [0, 1]) =0.$ 
 	\end{proof}

Before we present a variant of Theorem~\ref{coro-main}(ii) (in the case of constant sequence $c$ with value $1$), 
we record here the following consequence of \cite[Theorem~3.1]{AC2017} and \cite[Theorem~IV.14a]{W1941}$:$ 
 \begin{align} \label{uno-5}
   \begin{minipage}{67ex}
\text{\em For a polynomial $p : \mathbb R_+ \rar (0, \infty)$ with all real roots, $\big\{\frac{1}{p(m)}\big\}_{m \in \mathbb Z_+}$} \\ {\em   is a completely monotone sequence with the representing measure} \\ {\em being a weighted Lebesgue measure. In particular, $\big\{\frac{1}{p(m)}\big\}_{m \in \mathbb Z_+}$} \\ \text{\em  is a minimal completely monotone sequence.} 
 \end{minipage}
 \end{align}
 	
\begin{corollary} \label{main-thm-coro}
Let $a, b$ be as given in \eqref{a-and-b}.
If $\big\{\frac{1}{b(m) + a(m)n}\big\}_{m,n \in \mathbb{Z}_+}$ is a joint completely monotone net, then it is a minimal joint completely monotone net.
\end{corollary}
	\begin{proof} 
 	An examination of the proof of Theorem~\ref{main-thm}$(b)$ shows that (under the assumption that $\big\{\frac{c(m)}{b(m) + a(m)n}\big\}_{m,n \in \mathbb{Z}_+}$ is joint completely monotone) it depends only on \eqref{mu-j-new}, \eqref{weighted-L}, \eqref{rep-main-seq} and Lemma~\ref{convolution-mr}. On the other hand, it is recorded in Remark~\ref{assu-A1-used} that the assumption (A1) is not required in the deduction of \eqref{each-t-int}-\eqref{rep-main-seq}. 
 	Thus it suffices to check that 
 	\begin{enumerate}
 	\item[$\mathrm{(i)}$] $\mu_{j, t}([0, 1]) > 0,$ $j=1, \ldots, k,$ $t \in (0, 1),$ (so that Lemma~\ref{convolution-mr} applies),
 	\item[$\mathrm{(ii)}$] (A2) holds with $c=1.$
 	\end{enumerate}
 	The assertion (i) follows from \eqref{each-t-int}, while (ii) is immediate from  \eqref{uno-5}, completing the proof.  
	\end{proof}

Here is another consequence of the proof of Theorem~\ref{main-thm}.
\begin{corollary} \label{coro-proof}
Let $a, b$ be as given in \eqref{a-and-b}.
Let $j=1, \ldots, k$ and $t \in (0, 1).$  
If $(-1)^{j} b(-a_j) < 0,$ then $\Big\{t^{\frac{b(-a_j)}{d_j(m+a_j)}}\Big\}_{m \in \mathbb Z_+}$ is not a Hausdorff moment sequence, where $d_j = a_0 \prod_{1 \Le l \neq j \Le k}(a_l-a_j)$ $($we use here the convention that the product over empty set is $1).$ 
\end{corollary}
\begin{proof}
Assume that  $(-1)^{j} b(-a_j) < 0$ for some $j=1, \ldots, k.$ 
By 
\eqref{formula-c-j-Ramiz}, 
\beq \label{cj-positive}
c_j=-\frac{(-1)^jb(-a_j)}{a_0 \prod_{1 \Le l \Le j-1}(a_j-a_l) \prod_{j+1 \Le l \Le k}(a_l-a_j)} > 0.
\eeq
In view of \eqref{each-t-int}, it suffices to check that for every $t_0 \in (0, 1),$
$w_j(\cdot,t_0) < 0$ on some open subset of $(0, 1).$ 
For $t_0 \in (0, 1),$ letting 
 $s_0=e^{\frac{1}{c_j \log t_0}}$ in \eqref{form-w-j}, 
 we obtain 
\beqn
w_j(s_0,t_0) &=& 
t^{\frac{1}{k}\big(\frac{b_0}{a_0}-1\big)}_0 s^{a_j-1}_0 \sum_{l=1}^{\infty}\frac{(- \log s_0)^{l-1}}{(l-1)!l!} \, (c_j \log t_0)^l
\\ &=& c_j t^{\frac{1}{k}\big(\frac{b_0}{a_0}-1\big)}_0 s^{a_j-1}_0  (\log t_0) \sum_{l=1}^{\infty}\frac{(-1)^{l-1}}{(l-1)!l!},
\eeqn
which is less than $0$ by \eqref{cj-positive}. 
This shows that $w_j(s_0, t_0)<0.$ Since $w_j$ is continuous at $(s_0, t_0),$ the proof is complete. 
\end{proof}

\begin{corollary} \label{coro-main-thm}
Let $a, b$ be as given in \eqref{a-and-b} and assume that $(A1)$ holds. 
Let $F$ and $G$ be two $($possibly empty$)$ subsets of $\{1, \ldots, k\}$ and $\{2, \ldots, k\},$ respectively.  Let $c$ be any one of the following choices:  
\beqn
c(x) = \prod_{j \in F}(x+a_j), \quad c(x)= \prod_{j \in G}(x+b_j), \quad x \in \mathbb R_+.
\eeqn
Then $\big\{\frac{c(m)}{b(m) + a(m)n}\big\}_{m,n \in \mathbb{Z}_+}$ is a minimal joint  completely monotone net.  
\end{corollary}
\begin{proof} In view of Theorem~\ref{main-thm}, it suffices to check that $\big\{\frac{c(m)}{a(m)}\big\}_{m \in \mathbb Z_+}$ is a minimal completely monotone sequence. 
If $c(x)= \prod_{j \in F}(x+a_j),$ then \eqref{uno-5} (applied to $p=a/c$)  shows that $\big\{\frac{c(m)}{a(m)}\big\}_{m \in \mathbb Z_+}$ is a minimal completely monotone sequence. 
To see the conclusion in the second case, 
let $c(x)= \prod_{j \in G}(x+b_j).$
Write $G=\{i_1, \ldots, i_l\} $ and $G'=\{i_1-1,i_2-1,\ldots, i_l-1\}$  with $i_1 \Le \ldots \Le i_l.$
Note that 
\beq \label{c-a-expression}
\frac{c(m)}{a(m)} = \frac{1}{a_0}\,\frac{\prod_{j=1}^l(m+b_{i_j})}{\prod_{j =1}^k(m+a_j)} =  \frac{1}{d(m)} \frac{\prod_{j=1}^l(m+b_{i_j})}{\prod_{j \in G'}(m+a_{j})},
\eeq
where $d(m) = a_0 \prod_{\overset{j \notin  G'}{j=1, \ldots, k}}(m+a_{j}),$ $m \in \mathbb Z_+.$
  Consider
\beqn
  \frac{\prod_{j=1}^l(m+b_{i_j})}{\prod_{j \in G'}(m+a_{j})}= \frac{\prod_{j=1}^l(m+b_{i_j})}{\prod_{j=1}^l(m+a_{i_j-1})}, \quad m \in \mathbb Z_+.
\eeqn
By the assumption (A1), we see that  
$$
\sum_{j=1}^{p}b_{i_j}\Ge \sum_{j=1}^p a_{{ i_j}-1 },
\quad p = 1, \ldots, l.
$$ Hence, by \cite[Theorem~1]{Kball},  $ \frac{\prod_{j=1}^l(x+b_{i_j})}{\prod_{j=1}^l(x+a_{i_j-1})}$ is completely monotone function from $\mathbb R_+$ into $(0, \infty).$ Hence $\Big\{\frac{\prod_{j=1}^l(m+b_{i_j})}{\prod_{j \in G'}(m+a_{j})}\Big\}_{m \in \mathbb Z_+}$  is a minimal completely monotone sequence (see \cite[Theorem~IV.14b]{W1941}). 
\end{proof}

 We need the following general fact in the proof of Theorem~\ref{coro-main}(ii)$:$
\begin{lemma} 
\label{lem-necessary}
For polynomials $a, b : \mathbb R_+ \rar (0, \infty),$ 
let $p(x, y)=b(x)+a(x)y,$ $x, y \in \mathbb R_+.$  
If $\frac{1}{p}$ is a joint completely monotone function, then 
\beq \label{necessary-c}
a'(x)b(x) & \Le & a(x)b'(x), \quad x \in \mathbb R_+, \\
\label{necessary-c-4}
a(x_2)b(x_1) & \Le & a(x_1)b(x_2),
\quad x_2 \Ge x_1 \Ge 0. 
\eeq
and $\deg a \Le \deg b.$ 
\end{lemma}
\begin{proof} Assume that $\frac{1}{p}$ is a joint completely monotone function. 
A routine calculation using induction on $n \Ge 1$ shows that 
\beqn
 \partial^n_2\Big(\frac{1}{p}\Big)(x, y)=\frac{(-1)^nn! a(x)^n}{(b(x)+a(x)y)^{n+1}}, \quad n \in \mathbb Z_+, ~x, y \in \mathbb R_+.
\eeqn
Thus, for any positive integer $n,$ 
\beqn 
 & & (-1)^{n+1} \partial_1 \partial^n_2\Big(\frac{1}{p}\Big)(x, y) \\ &=& -n! \partial_1 \Big(\frac{a(x)^n}{(b(x)+a(x)y)^{n+1}}\Big) \notag \\ \notag 
& =& -n!\Big(\frac{na(x)^{n-1}a'(x)}{(b(x)+a(x)y)^{n+1}} -(n+1) \frac{a^{n}(x)(b'(x)+a'(x)y)}{(b(x)+a(x)y)^{n+2}} \Big) \\ \notag 
&=& \frac{n!a(x)^{n-1}}{p(x, y)^{n+2}} \Big((n+1)a(x)(b'(x)+a'(x)y) - na'(x)(b(x)+a(x)y)\Big).
\eeqn
This together with the joint complete monotonicity of $\frac{1}{p}$ yields
 \beqn
&& (-1)^{n+1}\partial_2 \partial^n_1\Big(\frac{1}{p}\Big)(x, y) \Ge 0  \Longrightarrow \\ & & (n+1)a(x)(b'(x)+a'(x)y) \Ge na'(x)(b(x)+a(x)y),  ~ n \Ge 1, ~x, y \in \mathbb R_+.
\eeqn
Letting $y= 0$ and dividing by $n,$ we get $$
(1+1/n)a(x)b'(x) \Ge a'(x)b(x), \quad n \Ge 1, ~x \in \mathbb R_+.$$
We now let $n \rar \infty$ to get \eqref{necessary-c}. 
To see \eqref{necessary-c-4}, note that by \eqref{necessary-c},
\beqn 
\frac{a'(x)}{a(x)} \Le \frac{b'(x)}{b(x)}, \quad x \in \mathbb R_+.
\eeqn
After integrating over $[x_1, x_2]$ and taking exponential on both sides, we get 
\eqref{necessary-c-4}.  
Letting $x_1 = 0$ and $x_2=m$ in \eqref{necessary-c-4}, we obtain 
\beqn
\frac{b(0)}{a(0)} \Le \frac{b(m)}{a(m)}, \quad m \Ge 0. 
\eeqn
This combined with $a(0) > 0$ and $b(0)>0$ yields $\deg a \Le \deg b.$ 
\end{proof}
\begin{remark} Assume that $\frac{1}{p}$ is a joint completely monotone function.
By \eqref{necessary-c}, $\Big(\frac{a}{b}\Big)' \Le 0$ for every $x \Ge 0.$ This combined with \eqref{p-frac} and \eqref{formula-c-j-Ramiz} (with roles of $a$ and $b$ interchanged) shows that
 \beq \label{nece-condition-6}
\sum_{j=1}^{k} \frac{\prod_{l=1}^k (a_l-b_j)}{\prod_{1 \Le l \neq j \Le k}(b_l-b_j)(x+b_j)^2} \Ge 0, \quad x \Ge 0.
 \eeq
 In particular, by multiplying on the left hand side by $x^{2},$ $x > 0,$ and letting $x \rar \infty,$ we obtain 
 $$\sum_{j=1}^{k} \frac{\prod_{l=1}^k (a_l-b_j)}{\prod_{1 \Le l \neq j \Le k}(b_l-b_j)} \Ge 0.$$ Moreover, we have
 $$\sum_{j=1}^{k} \frac{\prod_{l=1}^k (a_l-b_j)}{b^2_j\prod_{1 \Le l \neq j \Le k}(b_l-b_j)} \Ge 0.$$
This may be obtained by letting $x=0$ in \eqref{nece-condition-6}.
\hfill $\diamondsuit$
\end{remark}

The following lemma justifies the fact that a minimal completely monotone function is interpolated by its natural extension to $\mathbb R^2_+.$
\begin{lemma}\label{minimal-implies-JCM} 
    Let $a, b$ be as given in \eqref{a-and-b}. Then there exists a finite Radon measure $\nu$ on $\mathbb R^2_+$ such that
 \beqn
 \frac{1}{b(x)+a(x)y}=\int_{\mathbb R^2_+}e^{-(xt_1+yt_2)}d\nu(t_1,t_2), \quad x, y \in \mathbb R_+.
 \eeqn
 In particular, if the net $\left\{\frac{1}{b(m) + a(m)n}\right\}_{m,n \in \mathbb{Z}_+}$ is minimal joint completely monotone, then the function $\frac{1}{b(x)+a(x)y}$ is joint completely monotone.
\end{lemma}
\begin{proof} Note that a simple calculation as seen in \eqref{EQN4} shows that
		\beq \label{EQN4-2}
	\frac{1}{b(x) + a(x)y}
	&=& \int_{[0, 1]} t^y\,\frac{t^{\frac{b(x)}{a(x)}-1}}{a(x)}dt,  
	\quad x,y \in \mathbb{R}_+.
		\eeq
Since $\frac{a_0}{a(x)}$ is a finite product of completely monotone functions $\frac{1}{(x+a_j)}$, $j =1, \ldots, n$, we note that $\frac{1}{a}$ is joint completely monotone on $\mathbb{R}_+.$ Now by \cite[Theorem~1.1.4]{SSV}, there exists a unique Radon measure $\mu$ on $\mathbb R_+$ such that
\beq \label{1/a-meas}
\frac{1}{a(x)}=\int_{\mathbb R_+} e^{-xt}\mu(dt), \quad x \in \mathbb R_+.
\eeq
We argue as in the proof II of Theorem~\ref{main-thm}$(a)$ to see that, 
		\beq \label{s-power-prod-2}
		t^{\frac{b(x)}{a(x)}-1}=
	t^{\frac{b_0}{a_0}-1}\,	\prod_{j=1}^{k}t^{\frac{c_j}{x+a_j}}, \quad t >0, ~x \in \mathbb R_+,
		\eeq
  where $c_1, \ldots, c_n$ are given by \eqref{formula-c-j} and
\beqn
		t^{\frac{1}{k}\big(\frac{b_0}{a_0}-1\big)} t^{\frac{c_j}{x+a_j}} 
=	\int_{[0, 1]} s^x \mu_{j, t}(ds), \quad x \in \mathbb R_+, ~j=1. \ldots, k
		\eeqn
with $\mu_{j, t}$ given by \eqref{mu-j-new}. For every $t \in \mathbb(0,1),$ we now use the change of variable $s=e^{-u}$ to see
\beq \label{each-t-int-2}
t^{\frac{1}{k}\big(\frac{b_0}{a_0}-1\big)} t^{\frac{c_j}{x+a_j}}=\int_{\mathbb R_+}e^{-ux}\rho_{j,t}(du), \quad ~j=1. \ldots, k,
\eeq
where $\rho_{j,t}(du)=\mu_{j,t}(-e^{-u}du).$ It now follows from \eqref{EQN4-2}-\eqref{each-t-int-2} that
\beqn
		\frac{1}{b(x) + a(x)y}=\int_{[0, 1]} \int_{\mathbb R_+} e^{-ux}t^y \nu_t(du) dt, \quad x,y \in \mathbb{R}_+,
		\eeqn
		where $\nu_t$ is the multiplicative convolution of $\mu$ and $\rho_{j, t},$ $j=1, \ldots, k$. Again using the change of variable $t=e^{-v},$ we obtain
  \beqn
  	\displaystyle \frac{1}{b(x) + a(x)y}=\int_{\mathbb R_+} \int_{\mathbb R_+} e^{-ux}e^{-vy} (-e^{-v})\nu_{e^{-v}}(du) dv, \quad x,y \in \mathbb{R}_+.
  \eeqn
  This completes the proof of the first half.
	
Assume that the net $\left\{\frac{1}{b(m) + a(m)n}\right\}_{m,n \in \mathbb{Z}_+}$ is  minimal joint completely monotone. By \cite[Proposition~6]{Athavale2001}, there exists a joint completely monotone function $f$ on $\mathbb{R}_+^2$ such that
    \beq \label{f=1/p}
    f(m,n)=\frac{1}{b(m)+a(m)n}, \quad m,n \in \mathbb{Z}_+.
    \eeq
        In view of the proof \cite[Proposition~6]{Athavale2001}, there exists a positive Radon measure $\mu$ on $\mathbb R^2_+$ such that
    \beq \label{function -int}
    f(x,y)=\int_{\mathbb{R}_+^2}e^{-(xt_1+yt_2)}d\mu(t_1,t_2), \quad  x,y \in \mathbb{R}_+.
    \eeq
  By the first part, there exists a finite Radon
measure $\nu$ on $\mathbb R^2_+$ such that
   \beq \label{int-form-2}
   \frac{1}{b(x)+a(x)y}=\int_{\mathbb R^2_+}e^{-(xt_1+yt_2})d\nu(t_1,t_2), \quad x, y \in \mathbb R_+.
   \eeq
  It now follows from the uniqueness of the representing measure that $\mu=\nu,$ which completes the proof. 
\end{proof}

\begin{proof}[Proof of Theorem~\ref{coro-main}] (i) This is a special case of  Corollary \ref{coro-main-thm} (the case of $F = \emptyset$).

(ii) 
Assume that $\big\{\frac{1}{b(m) + a(m)n}\big\}_{m,n \in \mathbb{Z}_+}$ is a joint completely monotone net. 
By Corollary~\ref{main-thm-coro}, $\big\{\frac{1}{b(m) + a(m)n}\big\}_{m,n \in \mathbb{Z}_+}$ is a minimal joint completely monotone net. Hence, by Lemma~\ref{minimal-implies-JCM}, $\frac{1}{b(x)+a(x)y}$ is a completely monotone function, and consequently, Lemma~\ref{lem-necessary} is applicable. 
To obtain \eqref{necessary-c-2}, we let $x=0$ in \eqref{necessary-c} and simplify the expression. 
By \eqref{necessary-c-4}, 
\beqn
\frac{\prod_{j=1}^{k}  (x_1+b_j)}{\prod_{j=1}^{k} (x_1+a_j)} \Le \frac{\prod_{j=1}^{k}  (x_2+b_j)}{\prod_{j=1}^{k} (x_2+a_j)}, \quad x_2 \Ge x_1 \Ge 0. 
\eeqn 
Taking $x_2 \rar \infty$, we get
\beq \label{b-less-a}
\prod_{j=1}^{k}  (x_1+b_j) \Le \prod_{j=1}^{k} (x_1+a_j), \quad x_1 \Ge 0.
\eeq
Letting $x_1=0$, we get \eqref{necessary-c-3}.
After cancelling $x_1^k$ from both sides of \eqref{b-less-a}, dividing by $x_1^{k-1}$ and letting $x_1 \rar \infty$, we get \eqref{necessary-c-1}.
\end{proof}

\section{Some cases of lower bi-degree}

In this section, we present solutions to some special instances of Question~\ref{Q1} of lower bi-degree.
Before we state and prove the first result in this direction, recall that for a positive real number $\nu,$ 
the {\it Bessel function $J_{\nu}(z)$ of the first kind of order $\nu$} is given by
\beqn
J_{\nu}(z)=\Big(\frac{z}{2}\Big)^{\nu} \sum_{k=0}^\infty \Big(\frac{-z^{2}}{4}\Big)^k\frac{1}{k!\Gamma(\nu+k+1)}, \quad z \in \mathbb C \setminus (-\infty, 0],
\eeqn
where $\Gamma$ denotes the Gamma function. 
The {\it modified Bessel function $I_{\nu}(z)$ of the first kind of order $\nu$} is given by 
\beq \label{I-nu}
I_{\nu}(z)=\Big(\frac{z}{2}\Big)^{\nu} \sum_{k=0}^\infty \Big(\frac{z^{2}}{4}\Big)^k\frac{1}{k!\Gamma(\nu+k+1)}, \quad z \in \mathbb C \setminus (-\infty, 0],
\eeq
 (see \cite[Eqns~9.1.10 \& 9.6.10]{AS1992}). It is worth noting that the expression of $w_j,$ as given in \eqref{form-w-j}, takes the following form$:$ For any $s, t \in (0, 1),$
 \beqn
		w_j(s,t)  
	=  t^{\frac{1}{k}\big(\frac{b_0}{a_0}-1\big)} s^{a_j-1}  \frac{c_j \log t}{\sqrt{-c_j \log s \log t}} \, I_1\big(2\sqrt{-c_j \log s \log t}\big). 
\eeqn
\begin{theorem}[Degree at most $(1, 1)$] \label{deg-1-1} Let $p : \mathbb R^2_+\rightarrow (0, \infty)$ be a polynomial given by 
\beqn
p(x, y) = a + bx + cy + dxy, \quad x, y \in \mathbb R^2_+,
\eeqn  
where $a, b, c, d \in \mathbb R.$ 
The following statements are equivalent$:$
\begin{enumerate}
\item[$\mathrm{(i)}$] $\frac{1}{p}$ is a joint completely monotone function,
\item[$\mathrm{(ii)}$] $M:=bc-ad  \Ge 0,$
\item[$\mathrm{(iii)}$] $\big\{\frac{1}{p(m,n)}\big\}_{m,n \in \mathbb{Z}_+}$ is a joint completely monotone net,
\item[$\mathrm{(iv)}$] $\big\{\frac{1}{p(m,n)}\big\}_{m,n \in \mathbb{Z}_+}$ is a minimal joint completely monotone net.
\end{enumerate}
If $(ii)$ holds, then 
for every positive integer $l,$ $\big\{\frac{1}{p(m,n)^l}\big\}_{m,n \in \mathbb{Z}_+}$ is a minimal joint completely monotone net with the representing measure 
 given by
\beq \label{w-l-s-t}
\begin{cases}
\frac{s^{\frac{c}{d}-1}t^{\frac{b}{d}-1}}{d(l-1)!}  \left(\frac{\log t \log s}{M}\right)^{\frac{l-1}{2}} I_{l-1}\left(\frac{2}{d}\sqrt{M\log s \log t}\right)\,dsdt, & M > 0, ~d \neq 0, \\
\frac{s^{\frac{c}{d}-1} t^{\frac{b}{d}-1}}{d^l(l-1)!}  \frac{(\log t\log s)^{l-1}}{(l-1)!} \,dsdt, & M = 0, ~d \neq 0, \\
\frac{(-1)^{l-1}}{b^l (l-1)!} \log(s)^{l-1}s^{\frac{a}{b}-1} d\delta_{s^{c/b}}(t)\,ds, & b \neq 0, ~d=0.
\end{cases}
\eeq
Moreover, if $d \neq 0,$ $M \Ge 0$ and $\omega_{_{M, l}}$ denotes the weight function of the representing measure of $\big\{\frac{1}{p(m,n)^l}\big\}_{m,n \in \mathbb{Z}_+},$ then 
\beq \label{asymptote}
\lim_{\overset{M > 0}{M \rar 0}}\omega_{_{M, l}} = \omega_{_{0, l}}, \quad l \Ge 1. 
\eeq
\end{theorem}
\begin{proof}
Since $p(0, 0) = a$ and the image of $p$ is contained in $(0, \infty),$ $a$ is positive. After dividing $p(x, 0)$ by $x$ and letting $x \rar \infty,$ we conclude that $b$ is nonnegative. By symmetry, $c$ is also nonnegative.  Dividing $p(x, x)$ by $x^2$ and letting  $x \rar \infty,$ we see that $d$ is nonnegative. Thus
\beq \label{abcd-positive}
a >0, ~b \Ge 0, ~c \Ge 0, ~d \Ge 0.
\eeq
To see the equivalence of (i)-(iv), we show that
\begin{enumerate}
\item[] (iv)$\Longrightarrow$(iii)$\Longrightarrow$(ii)$\Longrightarrow$(iv)$\Longrightarrow$(i)$\Longrightarrow$(ii).
\end{enumerate}

(i)$\Rightarrow$(ii): Let $a(x)=c+dx,$ $b(x)=a+bx.$ By Lemma~\ref{lem-necessary}, $a'(0)b(0) \Le a(0)b'(0)$ or equivalently, $M \Ge 0.$

(iv)$\Rightarrow$(i): This follows from \cite[Proposition~6]{Athavale2001}.

(ii)$\Rightarrow$(iv): 
We first consider the case of $d \neq 0.$ Since $ad-bc \Le 0,$ by \eqref{abcd-positive}, 
$b$ and $c$ are positive. 
For $x,y \in \mathbb{R}_+,$ we write 
\beqn
p(x,y)=a+bx+cy+dxy=b(x+a/b)+d(x+c/d)y.
\eeqn
Since $\{\frac{1}{m+ c/d}\}_{m \in \mathbb Z_+}$ is a completely monotone sequence and  
$a/b \Le c/d,$ an application of  Theorem~\ref{coro-main}(i) (to $a(m)=c+dm$ and $b(m)=a+ bm,$ $m \in \mathbb Z_+$) shows that $\big\{\frac{1}{p(m,n)}\big\}_{m,n \in \mathbb{Z}_+}$ is a minimal joint completely monotone net. We now consider the case of $d=0.$ If both $b$ and $c$ are zero, then $\frac{1}{p},$ being a constant polynomial, is a joint completely monotone net. Now consider the case when either $b \neq 0$ or $c \neq 0.$ By the symmetry, we may assume that $b \neq 0.$ 
Note that  
\beqn 
 && \frac{1}{(a+bm+cn)^l}  \\ \notag
&& \overset{\eqref{EQN2}}=
\frac{(-1)^{l-1}}{b^l (l-1)!}\int_{[0, 1]}(\log s)^{l-1}s^{\frac{a+cn}{b}-1+m}ds, ~ m,n \in \mathbb{Z}_+.
\eeqn
Since $s^{\frac{cn}{b}}=\int_{0}^1 t^n d\delta_{s^{c/b}}(t)$ for $n \Ge 0$ and $s \in (0, 1),$ in case of $d \neq 0,$ the representing measure of $\big\{\frac{1}{p(m,n)}\big\}_{m,n \in \mathbb{Z}_+}$ is given by 
\eqref{w-l-s-t}. It now immediate from \cite[Proposition~5]{Athavale2001} that $\big\{\frac{1}{p(m,n)}\big\}_{m,n \in \mathbb{Z}_+}$ is a minimal joint completely monotone net. 

(iv)$\Rightarrow$(iii): Trivial.

Before we prove (iii)$\Rightarrow$(ii), let us see some general facts under the assumption that $d \neq 0.$
Let $l$ be a positive integer.   
Since $\big\{\frac{1}{p(m,n)}\big\}_{m,n \in \mathbb{Z}_+}$ is joint completely monotone, so is the net $\big\{\frac{1}{p(m,n)^l}\big\}_{m,n \in \mathbb{Z}_+}$ (see \cite[Lemma~8.2.1(v)]{BCR1984}). 
We now find the representing measure of $\big\{\frac{1}{p(m,n)^l}\big\}_{m,n \in \mathbb{Z}_+}.$   
Note that for $m,n \in \mathbb{Z_+},$
\allowdisplaybreaks
\beq
\label{eq:someequation}
\notag
\frac{1}{p(m,n)^l} 
& = & \frac{1}{(b+dn)^l(\frac{a+cn}{b+dn}+m)^l}\\ \notag
 & \overset{\eqref{EQN2}}= & \frac{1}{(b+dn)^l}\frac{(-1)^{l-1}}{(l-1)!}\int_{[0, 1]} (\log s)^{l-1}s^{\frac{a+cn}{b+dn}-1+m}ds\\   
& = & \int_{[0, 1]} s^{m} \Big(\frac{1}{(b+dn)^l} \, s^{\frac{a- \frac{bc}{d}}{b+dn}}\Big) \frac{(-\log s)^{l-1}s^{\frac{c}{d}-1}}{(l-1)!} \, ds.
\eeq
Further, for any $s >0$ and $n \in \mathbb Z_+,$  
\beqn 
& & \frac{1}{(b+dn)^l} \, s^{\frac{a-\frac{bc}{d}}{b+dn}} \\
 & = & \frac{1}{(b+dn)^l}\, e^{\frac{a-\frac{bc}{d}}{b+dn}\log s}  \\
 & = & \sum_{k=0}^{\infty}\frac{1}{k!}\frac{(a-\frac{bc}{d})^{k}}{(b+dn)^{k+l}}\, (\log{s})^k \\
 & \overset{\eqref{EQN2}} = & \sum_{k=0}^{\infty}\frac{(a-\frac{bc}{d})^{k}}{d^{k+l}k!(k+l-1)! } \, (\log{s})^k \int_{[0, 1]}(-\log t)^{k+l-1}t^{\frac{b}{d}-1+n}dt \\
& = & \int_{[0, 1]} t^n \Bigg(\sum_{k=0}^{\infty}\frac{(\frac{bc}{d}-a)^{k}}{d^{k+l}k!(k+l-1)!} (-\log{s})^k(-\log t)^{k+l-1}t^{\frac{b}{d}-1}\Bigg)dt,
\eeqn
where we can interchange the series and the integral by the dominated convergence theorem. 
This combined with \eqref{eq:someequation} shows that 
\beqn
\frac{1}{p(m,n)^l}=\int_0^1\int_{[0, 1]} s^m t^n \omega_l(s,t)dt ds,  \quad m,n \in \mathbb{Z_+},
\eeqn 
where  $\omega_l$ is as given in \eqref{w-l-s-t}.
\begin{align} \label{wt-d-nonzero}
\omega_l(s,t)=\frac{s^{\frac{c}{d}-1} t^{\frac{b}{d}-1}}{d^l(l-1)!}  \sum_{k=0}^{\infty}\left(\frac{M}{d^2}\right)^{k} \frac{(\log t\log s)^{k+l-1}}{k!(k+l-1)!}, \quad s, t \in (0, 1).
\end{align}
If $M \Ge 0,$ then the above discussion together with \eqref{I-nu} yields \eqref{w-l-s-t} and~\eqref{asymptote}.

(iii)$\Rightarrow$(ii): 
Assume that $M = bc-ad <0.$ By \eqref{abcd-positive}, $d > 0.$
Hence, by \eqref{wt-d-nonzero},  
\beqn \omega_1(s, t) &=& \frac{s^{\frac{c}{d}-1} t^{\frac{b}{d}-1}}{d}  \sum_{k=0}^{\infty}\left(\frac{M}{d^2}\right)^{k} \frac{(\log t\log s)^{k}}{(k!)^2} \\
&=& \frac{s^{\frac{c}{d}-1}t^{\frac{b}{d}-1}}{d}  J_{0}\left(\frac{2}{d}\sqrt{-M\log s \log t}\right), \quad s, t \in (0, 1).
\eeqn
However, the Bessel function $J_0$ takes negative values on some open interval in $(0, \infty)$ (see~\cite[Eqn~9.1.18]{AS1992}), and hence 
 $\omega_1$ takes negative values on some open subset of $(0, 1) \times (0, 1).$ It follows that  
$\big\{\frac{1}{p(m,n)}\big\}_{m,n \in \mathbb{Z}_+}$ is not a joint completely monotone net. 
This yields (iii)$\Rightarrow$(ii), completing the proof.  
\end{proof}

\subsection*{Case of bi-degree $(2, 1)$}

A nonconstant polynomial $p : \mathbb R^2_+ \rar (0, \infty)$ of bi-degree $(2, 1)$ is given by $p(x,y)=b(x)+a(x)y$. Note that $p(x,0)=b(x)$, and hence $b$ is a mapping from $\mathbb R_+$ into $(0,\infty).$ If $a(x_0)<0$ for some $x_0 \in \mathbb R_+,$ then for large value of $y$, $p(x_0,y)<0$, and hence $a$ maps  $\mathbb R_+$ into $\mathbb R_+.$ 
Assume that $\frac{1}{p}$ is joint completely monotone. 
By Example \ref{exam-1.3-i}(i), $a$ is never constant. Also, by Lemma~\ref{lem-necessary}, $\deg a \Le \deg b.$
Furthermore, we have the following$:$
\begin{enumerate}
\item[$\bullet$] 
If $b$ has a complex root, then $\frac{1}{p(\cdot,0)}$ is not completely monotone (see \cite[Proposition~4.3]{AC2017}), and hence $b$ has negative real roots. 
\item[$\bullet$] 
If possible, assume that $a$ has a complex root. Then $a$ maps
$\mathbb{R}$ into $(0, \infty),$ and hence for any $x \in \mathbb{R},$ we have $b(x)+a(x)y > 0$ for sufficiently large $y.$ This shows $b+ay$ has a complex root, and by another application of \cite[Proposition~4.3]{AC2017}, $\frac{1}{p}$ is not separately monotone. This together with Remark~\ref{JCM-rmk} leads to a contradiction that $\frac{1}{p}$ is not joint completely monotone. Hence $a$ has nonpositive real roots. 
\end{enumerate}

We summarize below the discussion in the preceding paragraph$:$

\begin{proposition} \label{prop-deg-2-1}
Consider a nonconstant polynomial $p : \mathbb R^2_+ \rar (0, \infty)$ of bi-degree $(2, 1)$ given by $p(x,y)=b(x)+a(x)y.$ Assume that $\frac{1}{p}$ is joint completely monotone. Then the following statements are valid$:$
\begin{enumerate}
\item[$\mathrm{(i)}$] $a$ maps $\mathbb R_+$ into $\mathbb R_+,$ $1 \Le \deg a \Le 2$ and $a$ has nonpositive real roots,
\item[$\mathrm{(ii)}$] $b$ maps $\mathbb R_+$ into $(0, \infty),$ $\deg b =2$ and $b$ has negative real roots.
\end{enumerate}
In particular, the following possibilities occur$:$
\begin{enumerate}
\item[$\mathrm{(a)}$] $p(x,y)=b_0 (x+b_1)(x+b_2) + a_0 (x+a_1)(x+a_2)y,$
\item[$\mathrm{(b)}$] $p(x,y)=b_0 (x+b_1)(x+b_2) + a_0 (x+a_1)y,$
\end{enumerate}
where $a_0, b_0, b_1, b_2 > 0$ and $a_1, a_2 \Ge 0.$ 
\end{proposition}

We discuss below the first sub-case of bi-degree $(2, 1),$ which is an outcome of a careful examination of the proof Theorem~\ref{main-thm}.
\begin{theorem}[Subcase (a) of bi-degree $(2, 1)$] \label{thm-bi-deg-2-1}
For $a_j, b_j \in (0, \infty),$ $j=0, 1,2,$ let  
$a(x)=a_0(x+a_1)(x+a_2)$ and $b(x)=b_0(x+b_1)(x+b_2), x\in \mathbb{R}_+$ with $a_1 \Le a_2$ and $b_1 \Le b_2.$  Then the following statements are valid$:$
\begin{enumerate}
\item[$\mathrm{(i)}$] $\big\{\frac{1}{b(m)+a(m)n}\big\}_{m,n \in \mathbb{Z}_+}$ is a minimal joint completely monotone net provided 
 \beq \label{assum-2-deg} 
 (b_1\Le a_1 \Le b_2~\text{or}~b_1\Le a_2 \Le b_2)~ \text{and}~ b_1+b_2 \Le a_1+a_2,
 \eeq
\item[$\mathrm{(ii)}$] if $\big\{ \frac{1}{b(m)+a(m)n}\big\}_{m,n \in \mathbb{Z}_+}$ is joint completely monotone, then 
\beqn
\frac{1}{a_1} + \frac{1}{a_2} \Le \frac{1}{b_1} + \frac{1}{b_2}, \quad b_1+b_2 \Le a_1+a_2.
\eeqn
\end{enumerate}
\end{theorem}

We need a lemma in the proof of Theorem~\ref{thm-bi-deg-2-1}.
\begin{lemma}\label{two-vari}
Let $c_1,c_2,b_1,b_2\in \mathbb{R}$ with $b_1 < b_2$. The following statements are equivalent$:$ 
\begin{enumerate}
\item[$\mathrm{(i)}$] for every $s \in (0,1),$ $c_1s^{b_1}+c_2s^{b_2}\Le 0,$
\item[$\mathrm{(ii)}$] $c_1 \Le 0$ and $c_1+c_2 \Le 0.$
\end{enumerate} 
\end{lemma}
\begin{proof} For $s \in (0,1),$ note that 
	\beq \label{basic-eq}
	c_1s^{b_1}+c_2s^{b_2}=s^{b_1}(c_1+c_2s^{b_2-b_1}).
	\eeq
(i)$\Rightarrow$(ii) 
	Since for every $s \in (0,1),$ $s^{b_1} > 0$ and $c_1s^{b_1}+c_2s^{b_2}\Le 0$, we must have $c_1+c_2s^{b_2-b_1}\Le 0$ for every $s \in (0, 1).$ Since $b_2 - b_1 > 0,$ letting $s \rar 0$, we obtain $c_1 \Le 0,$ and letting $s \rar 1,$ we obtain $c_1+c_2 \Le 0$. 
	
(ii)$\Rightarrow$(i) If $c_2 \Le 0$ then clearly for every $s \in (0,1),$ $c_1s^{b_1}+c_2s^{b_2}\Le 0.$ So we may assume that $c_2 > 0.$  In view of \eqref{basic-eq}, it suffices to check that $c_1+c_2s^{b_2-b_1}\Le 0$ for every $s \in (0, 1).$ Indeed, since $b_2 - b_1 > 0,$ 
	\beqn
	c_1+c_2s^{b_2-b_1} \Le c_1+c_2\Le 0, \quad s\in (0,1).
	\eeqn
	This completes the proof.
\end{proof}

\begin{proof}[Proof of Theorem~\ref{thm-bi-deg-2-1}] 
Assume that \eqref{assum-2-deg} holds. 
If  $b_1\Le a_2 \Le b_2$ and $b_1+b_2 \Le a_1+a_2,$ then  $b_1\Le a_1 \Le b_2.$ Indeed, since $a_1 \Le a_2 \Le b_2,$ if $a_1 < b_1,$ then $a_1 + a_2 < b_1 + b_2.$ Hence, without loss of generality, we may assume that $b_1\Le a_1 \Le b_2.$

If $b_1=b_2,$ then $a_1=b_1$ and $b_2 \Le a_2.$ The desired conclusion in this case now follows from Theorem~\ref{coro-main}(i). Hence, we may assume that $b_1 < b_2.$
We imitate the proof of Theorem~\ref{main-thm}. Indeed, in view the discussion following \eqref{p-frac}, it suffices to check that 
\beq \label{c-j-s}
c_1s^{b_1}+c_2s^{b_2}\Le 0, \quad s \in (0,1),
\eeq
where $c_1$ and $c_2$ are given by \eqref{formula-c-j}. 
We consider two cases$:$

$a_1 < a_2$: In view of  Lemma~\ref{two-vari}, it suffices to check 
that $c_1 \Le 0$ and $c_1 + c_2 \Le 0.$ 
By \eqref{formula-c-j}, 
$$c_1=\frac{b_0}{a_0} \frac{(b_1-a_1)(b_2-a_1)}{(a_2-a_1)}, \quad c_2=-\frac{b_0}{a_0} \frac{(b_1-a_2)(b_2-a_2)}{(a_2-a_1)}.$$ Hence,
$c_1 \Le 0$ and $c_1 + c_2 \Le 0$  if and only if 
\beqn
(b_1-a_1)(b_2-a_1) \Le 0, \quad (b_1-a_1)(b_2-a_1) \Le (b_1-a_2)(b_2-a_2).
\eeqn
This is easily seen to be equivalent to   
\beqn
(b_1-a_1)(b_2-a_1) \Le 0, \quad (b_1+b_2)(a_2-a_1) \Le (a_1+a_2)(a_2-a_1).
\eeqn
Both these inequalities follow at once from the assumption \eqref{assum-2-deg}.

$a_1=a_2$: By the partial fraction, 
\beqn
\frac{(m+b_1)(m+b_2)}{(m+a_1)^2}=1 + \frac{c_1}{m+a_1}+ \frac{c_2}{(m+a_1)^2}, \quad m \in \mathbb{Z}_+,
\eeqn
where $c_1=b_1+b_2-2a_1$ and $c_2 = (b_1-a_1)(b_2-a_1).$ By \eqref{assum-2-deg}, $c_1 \Le 0$ and $c_2 \Le 0,$ and hence we get \eqref{c-j-s}. The fact that $\big\{\frac{1}{b(m)+a(m)n}\big\}_{m,n \in \mathbb{Z}_+}$ is a minimal joint completely monotone net now follows from Corollary~\ref{main-thm-coro}.

Part (ii) is a special case of Theorem~\ref{coro-main}(ii).
\end{proof}
\begin{remark} \label{following} 
Assume that \eqref{assum-2-deg} holds. 
It follows from Theorem~\ref{thm-bi-deg-2-1}(i) that $\big\{\frac{1}{b(m)+a(m)n}\big\}_{m,n \in \mathbb{Z}_+}$ is a minimal joint completely monotone net. This combined with \cite[Proposition~6]{Athavale2001} shows that $\frac{1}{p}$ is a joint completely monotone function. 
\hfill $\diamondsuit$
\end{remark}

Here is the second sub-case of bi-degree $(2, 1).$
\begin{theorem}[Subcase (b) of bi-degree $(2, 1)$]  \label{bi-deg-2-1-second}
For $a_0, a_1, b_0, b_1, b_2 \in (0,\infty),$ let  
$a(x)=a_0(x+a_1)$ and $b(x)=b_0(x+b_1)(x+b_2),~ x\in \mathbb{R}_+.$ If $b_1 \Le a_1 \Le b_2,$ then $\big\{ \frac{1}{b(m)+a(m)n}\big\}_{m,n \in \mathbb{Z}_+}$ is joint completely monotone. 
Conversely, if $\big\{ \frac{1}{b(m)+a(m)n}\big\}_{m,n \in \mathbb{Z}_+}$ is joint completely monotone, then $\frac{1}{a_1} \Le \frac{1}{b_1}  + \frac{1}{b_2}.$
\end{theorem}
\begin{proof} Assume that $b_1 \Le a_1 \Le b_2.$ If $b_1 = a_1,$ then 
\beqn
\frac{1}{b(m)+a(m)n}=\frac{1}{m+a_1} \frac{1}{b_0m + a_0n + b_0b_2}, \quad m, n \in \mathbb Z_+,
\eeqn
and hence the desired conclusion in this case follows from Theorem~\ref{deg-1-1} and \cite[Lemma~8.2.1(v)]{BCR1984}. The same argument shows that $\big\{ \frac{1}{b(m)+a(m)n}\big\}_{m,n \in \mathbb{Z}_+}$ is joint completely monotone if $a_1=b_2.$ 

To complete the proof of the necessity part, 
let $c_0 = \frac{b_0}{a_0}$ and note that 
\beqn
\frac{b(x)}{a(x)} &=&  c_0 \frac{(x+b_1)(x+b_2)}{(x+a_1)} \\
&=& c_0(x+b_1+b_2-a_1) + c_0 \frac{(b_1-a_1)(b_2-a_1)}{x+a_1}, \quad x \in \mathbb R_+.
\eeqn
It follows that 
\beq \label{expree-t-b-by-a}
t^{\frac{b(m)}{a(m)}}=t^{c_0m} t^{c_0(b_1+b_2-a_1)}t^{c_0\frac{(b_1-a_1)(b_2-a_1)}{m+a_1}}, \quad m \in \mathbb{Z}_+.
\eeq
Now assume that $b_1< a_1< b_2$ and note that $(b_1-a_1)(b_2-a_1)<0.$
One may now argue as in the proof II of Theorem~\ref{main-thm}$(a)$ (see \eqref{each-t-int}-\eqref{form-w-j}) to show that for every $t \in (0, 1),$
$\big\{t^{c_0\frac{(b_1-a_1)(b_2-a_1)}{m+a_1}}\big\}_{m \Ge 0}$ is a 
Hausdorff moment sequence. Since $\{t^{c_0m}\}_{m \Ge 0}$ is a Hausdorff moment sequence for every $t \in (0, 1),$
 Lemma~\ref{continuum} (see \eqref{uno-5}) together with \eqref{expree-t-b-by-a} completes the proof of the sufficiency part. 

Let $c_1=c_0(b_1-a_1)(b_2-a_1).$ Arguing as in \eqref{bidegree-second-case}, we conclude that for $t \in (0,1),$ we have
\beqn
 \frac{t^{\frac{c_1}{(m+a_1)}}}{m+a_1} &=& \sum_{l=0}^{\infty}\frac{(c_1 \log t)^{l}}{{l!} (m+a_1)^{l+1}} \\
	&\overset{\eqref{EQN2}}=& 
	 \sum_{l=1}^{\infty} \frac{ (c_1 \log t)^l }{(l!)^2}  \int_{[0, 1]} (-\log {s})^{l}s^{a_1-1+m}ds. 
\eeqn
This combined with \eqref{expree-t-b-by-a} yields 
\beqn
 \frac{t^{\frac{b(m)}{a(m)}}}{a(m)}
= t^{c_0m} t^{c_0(b_1+b_2-a_1)} \sum_{l=1}^{\infty} \frac{ (c_1 \log t)^l }{a_0(l!)^2}  \int_{[0, 1]} (-\log {s})^{l}s^{a_1-1+m}ds, \\
m \in \mathbb{Z}_+, ~t \in (0, 1).
\eeqn
Thus, for some integrable function $w(s, t)$ (by an application of the dominated convergence theorem), we have 
\beqn
 \frac{t^{\frac{b(m)}{a(m)}}}{a(m)}
=\int_0^1 t^{c_0m}s^m w(s,t)ds = \int_0^1 (t^{c_0}s)^m w(s,t)ds, \quad m \in \mathbb Z_+, ~t \in (0, 1).
\eeqn 
Now using change of variable $st^{c_0}=s_1$ and applying \eqref{EQN4}, we may conclude that the representing measure of $\big\{ \frac{1}{b(m)+a(m)n}\big\}_{m,n \in \mathbb{Z}_+}$ is a weighted Lebesgue measure. 
One may now argue as in the proof of Lemma~\ref{minimal-implies-JCM} to see that $\frac{1}{b(x)+a(x)y}$  is a completely monotone function. The desired conclusion now follows from Lemma~\ref{lem-necessary} by letting $x=0$ in \eqref{necessary-c}. 
\end{proof}

\section{The Cauchy dual subnormality problem for toral $3$-isometric shifts}

Let $n$ be a positive integer and let $H$ be a complex Hilbert space. We say that $T$ is {\it commuting $n$-tuple on $H$} if $T_1, \ldots, T_n$ are bounded linear operators on $H$ such that $T_i T_j = T_j T_i$ for every $1 \Le i \neq j \Le n.$
Let $m$ be a positive integer.
Following \cite{Ag1990, Athavale2001, R1988}, we say that a commuting $n$-tuple $T$ is a {\it toral $m$-isometry} if 
\beqn 
\sum_{\overset{\alpha \in \mathbb Z^n_+}{0 \Le \alpha \Le \beta}} (-1)^{|\alpha|} \binom{\beta}{\alpha}T^{*\alpha}T^{\alpha} = 0, \quad \beta \in \mathbb Z^n_+, ~|\beta|=m,
\eeqn
where $T^{\alpha}$ stands for the bounded linear operator $\prod_{j=1}^nT^{\alpha_j}_j$ and $T^{*\alpha}$ denotes the Hilbert space adjoint of $T^{\alpha}.$ 
\begin{definition} \label{def-sep-m-iso}
We say that a commuting $n$-tuple $T$ is a {\it separate $m$-isometry} if $T_1, \ldots, T_n$ are $m$-isometries. 
\end{definition}

\begin{remark} For any commuting $n$-tuple $T,$ 
\beq 
\label{diff-op-m-iso}
\triangle^{\beta}(T^{*\gamma}T^{\gamma})\big|_{\gamma =0} = \sum_{\overset{\alpha \in \mathbb Z^n_+}{0 \Le \alpha \Le \beta}} (-1)^{|\alpha| + |\beta |} \binom{\beta}{\alpha}T^{*\alpha}T^{\alpha}, \quad \beta \in \mathbb Z^n_+.
\eeq
The verification of this identity is similar to that of \cite[Eqn~(2.1)]{JJS2020}. 
Clearly, a toral $m$-isometry is a separate $m$-isometry. Indeed, for any $j=1, \ldots, n,$ letting $\beta =m\varepsilon_j$ in \eqref{diff-op-m-iso} shows that $T_j$ is an $m$-isometry. In general, a separate $m$-isometry is not a toral $m$-isometry (see Remark~\ref{sep-toral-ex} below). 
 \hfill $\diamondsuit$
\end{remark}

The following is a well known fact for a single operator (see \cite[Equation~(1.3)]{AS1995}, \cite[Corollary~3.5]{JJS2020}).
\begin{theorem} \label{thm-char-m-iso}
 A commuting $n$-tuple $T=(T_1, \ldots, T_n)$ on $H$ is a toral $m$-isometry if and only if 
 \beqn 
 T^{*\alpha}T^{\alpha} = \sum_{k=0}^{m-1} \sum_{\overset{\beta \in \mathbb Z^n_+}{|\beta|=k}} \frac{\triangle^\beta(T^{*\alpha}T^{\alpha})\big|_{\alpha =0} }{\beta!} \,(\alpha)_\beta, \quad \alpha \in \mathbb Z^n_+.
 \eeqn
 \end{theorem}
 \begin{proof} Define $\mf m(\alpha)= T^{*\alpha}T^{\alpha},$ $\alpha \in \mathbb Z^n_+.$ By \eqref{diff-op-m-iso}, 
 \beq \label{equi-m-iso}
\mbox{$T$ is a toral $m$-isometry $\Leftrightarrow$ $\triangle^{\beta}\mf m=0,$ $\beta \in \mathbb Z^n_+,$ $|\beta| \Ge m.$} 
\eeq
 By the Newton's Interpolation Formula in several variables (see \cite[Remark~2, Equation~(G)]{Athavale2001}),
 for every $h \in H,$
\beqn
 \inp{\mf m(\alpha)h}{h} = \sum_{k=0}^{\infty} \sum_{\overset{\beta \in \mathbb Z^n_+}{|\beta|=k}} \frac{\triangle^\beta(\inp{\mf m(\alpha)h}{h})|_{\alpha =0} }{\beta!} \,(\alpha)_\beta, \quad \alpha \in \mathbb Z^n_+.
 \eeqn
 This combined with
 \beqn
\triangle^\beta \inp{\mf m(\cdot)h}{h}= \inp{\triangle^\beta \mf m(\cdot)h}{h}, \quad \beta \in \mathbb Z^n_+, ~h \in H, 
\eeqn
yields the following identity$:$ 
\beqn
 \inp{\mf m(\alpha)h}{h} = \sum_{k=0}^{\infty} \sum_{\overset{\beta \in \mathbb Z^n_+}{|\beta|=k}} \frac{\inp{\triangle^\beta (\mf m(\alpha))h}{h}|_{\alpha =0} }{\beta!} \, (\alpha)_\beta, \quad \alpha \in \mathbb Z^n_+, ~h \in H.
 \eeqn
This together with \eqref{equi-m-iso}
gives the necessity part. 
The sufficiency part follows from \eqref{fact-m-iso} and \eqref{equi-m-iso}. 
 \end{proof}
 
 The following, a special case of \cite[Proposition~4.1(iii)]{CC2012}, is immediate from Theorem~\ref{thm-char-m-iso}.
\begin{corollary} \label{last-coro}
A commuting $n$-tuple $T=(T_1, \ldots, T_n)$ on $H$ is a toral $2$-isometry if and only if 
\beqn
T^{*\alpha}T^{\alpha} = I + \sum_{j=1}^n \alpha_j (T^*_jT_j-I), \quad \alpha=(\alpha_1, \ldots, \alpha_n) \in \mathbb Z^n_+.
\eeqn
\end{corollary}

The next result characterizes separate $2$-isometries within the class of toral $m$-isometric pairs.
\begin{corollary} \label{coro-char-m-iso}
Let $m$ be an integer bigger than $1.$ A toral $m$-isometry pair  $T=(T_1, T_2)$ on  $H$ is a separate $2$-isometry if and only if 
\beq \label{conclusion-3-iso}
 T^{*\alpha}T^{\alpha} 
= 
I + \alpha_1 A + \alpha_2 \big(B + \alpha_1 C\big), \quad \alpha = (\alpha_1, \alpha_2) \in \mathbb Z^2_+. 
\eeq
where $A = T^*_1T_1-I,$ $B = T^*_2T_2-I$ and $C =T^{*}_1 B T_1 - B.$
If \eqref{conclusion-3-iso} holds, then $T$ is a toral $3$-isometry. 
\end{corollary}
\begin{proof} If $m=2,$ then the desired equivalence follows from Corollary~\ref{last-coro} (the identity \eqref{conclusion-3-iso} holds with $C=0$). Hence, we may assume that $m \Ge 3.$ Suppose that $T$ is a separate $2$-isometry. 
Define $\mf m(\alpha)=T^{*\alpha}T^{\alpha},$ $\alpha \in \mathbb Z^2_+.$
Since $T_2$ is a $2$-isometry, by \eqref{equi-m-iso} (applied to $m=2$ and $n=1$), $\triangle^j_2 (T^{*\alpha_2}_2T^{\alpha_2}_2)=0$ for every $j \Ge 2.$ 
This combined with Theorem~\ref{thm-char-m-iso} yields 
 \beq \label{i-step}
 \mf m(\alpha) &=& \sum_{j=0}^{m-1} \sum_{\overset{\beta \in \mathbb Z^2_+}{\beta_1 + \beta_2=j}} \frac{\triangle^{\beta_1}_1 \triangle^{\beta_2}_2(\mf m(\alpha))|_{\alpha =0} }{\beta_1! \beta_2!}\, (\alpha_1)_{\beta_1} (\alpha_2)_{\beta_2}\\ \notag
 &=& \sum_{j=0}^{m-1}  \frac{\triangle^{j}_1(\mf m(\alpha))|_{\alpha =0} }{j !} \, (\alpha_1)_{j}  + \sum_{j=1}^{m-1}   \frac{\triangle^{j-1}_1 \triangle_2(\mf m(\alpha))|_{\alpha =0} }{(j-1) !}\, (\alpha_1)_{j-1}\, \alpha_2
 \eeq
 for every $\alpha \in \mathbb Z^2_+.$ However, since $T_1$ is a $2$-isometry, once again by \eqref{equi-m-iso},
 \beqn
 \triangle^{j}_1(\mf m(\alpha))|_{\alpha =0} &=& \triangle^{j}_1(T^{*\alpha_1}_1T^{\alpha_1}_1)|_{\alpha_1 =0} = 0, \quad j \Ge 2, \\ \triangle^{j-1}_1 \triangle_2(\mf m(\alpha))|_{\alpha =0} &=& \triangle^{j-1}_1(T^{*\alpha_1}_1(T^*_2T_2-I)T^{\alpha_1}_1)|_{\alpha_1 =0} = 0, \quad j \Ge 3.
 \eeqn
 This combined with \eqref{i-step} yields
 \beqn 
\mf m(\alpha) &=& \sum_{j=0}^{1}  \frac{\triangle^{j}_1(T^{*\alpha_1}_1T^{\alpha_1}_1)|_{\alpha_1 = 0}}{j !} \,(\alpha_1)_{j}  \\ &+& \alpha_2 \sum_{j=0}^{1}   \frac{\triangle^{j}_1(T^{*\alpha_1}_1(T^*_2T_2-I)T^{\alpha_1}_1)|_{\alpha_1 = 0} }{j !}\, (\alpha_1)_{j} \\
&=& I +  \alpha_1 \triangle_1(T^{*\alpha_1}_1T^{\alpha_1}_1)|_{\alpha_1 = 0} + \alpha_2 ( B + \alpha_1 \triangle_1(T^{*\alpha_1}_1 B T^{\alpha_1}_1)|_{\alpha_1 = 0} ).
\eeqn
This gives the necessity part. The sufficiency part follows from applications of Corollary~\ref{last-coro} separately to $T_1$ and $T_2.$ The remaining part now follows from Theorem~\ref{thm-char-m-iso}. 
\end{proof}

Let ${\bf w} = \big\{{w^{(j)}_\alpha} : j=1, \ldots, n, ~\alpha \in {\mathbb Z}^n_+\big\}$ be a set of nonzero complex numbers. 
Let $\mathscr H$ be a complex separable Hilbert space with orthonormal basis $\mathscr E = \{e_\alpha\}_{\alpha \in \mathbb Z^n_+}.$
A {\it weighted $n$-shift} $\,\mathscr W = (\mathscr W_1, \ldots, \mathscr W_n)$ with respect
to $\mathscr E$ is defined by \beqn \mathscr W_j e_\alpha \mathrel{\mathop:}=
w^{(j)}_\alpha e_{\alpha + \varepsilon_j}, \quad j=1, \ldots, n, \eeqn where
$\varepsilon_j$ is the $m$-tuple with $1$ in the $j$th place and
  zeros elsewhere. Clearly, $\mathscr W_1, \ldots, \mathscr W_n$ extend to densely defined linear operators on the linear span of $\{e_\alpha\}_{\alpha \in \mathbb Z^n_+}.$ Note that $\mathscr W_1, \ldots, \mathscr W_n$ extend boundedly to $\mathscr H$  if and only if $\sup_{\alpha \in \mathbb Z^n_+}|w^{(j)}_\alpha| < \infty$ for every $j=1, \ldots, n.$ It is easy to see that for $i, j =1, \ldots, n,$ 
\beq \label{commuting}
\mathscr W_i\mathscr W_j = \mathscr W_j \mathscr W_i~ \Longleftrightarrow w^{(i)}_\alpha w^{(j)}_{\alpha
    + \varepsilon_i}=w^{(j)}_\alpha w^{(i)}_{\alpha + \varepsilon_j}, \quad \alpha \in {\mathbb Z}^n_+.
    \eeq
  The reader is referred to \cite{JL1979} for the basic theory of weighted multishifts. 
  
 In what follows, we always assume that {\it the weight multi-sequence ${\bf w}$ of $\mathscr W$ consists of
positive numbers, $\mathscr W$ extends boundedly to $\mathscr H$ and that  $\mathscr W$ is a commuting $n$-tuple.} 
 We indicate the weighted $n$-shift  $\mathscr W$
  with weight multi-sequence ${\bf w}$ by $\mathscr W : \{{w^{(j)}_\alpha}\}.$

 \begin{proposition} \label{m-iso-wt-char}
For a weighted $2$-shift $\mathscr W : \{w^{(j)}_\alpha\},$ the following statements are valid$:$
 \begin{enumerate}
\item[$\mathrm{(i)}$] $\mathscr W$ is a toral $m$-isometry if and only if 
 \beq \label{wt-2-iso}
 \|\mathscr W^\alpha e_0\|^2  = \sum_{k=0}^{m-1} \sum_{\overset{\beta \in \mathbb Z^2_+}{|\beta|=k}} \frac{\triangle^\beta(\|\mathscr W^{\alpha}e_0\|^2)|_{\alpha =0} }{\beta!} \,(\alpha)_\beta, \quad \alpha=(\alpha_1, \alpha_2) \in \mathbb Z^2_+,
 \eeq
 \item[$\mathrm{(ii)}$] 
 if \eqref{wt-2-iso} holds, then $\mathscr W$ is a separate $2$-isometry if and only if 
 \beq \label{wt-3-iso-new}
 \|\mathscr W^\alpha e_0\|^2 
= 
1 + \alpha_1 b + 
 \alpha_2( c + \alpha_1 d), \quad \alpha=(\alpha_1, \alpha_2) \in \mathbb Z^2_+,
\eeq
where $b, c, d$ are given by 
\beq
\label{value-d}
\left.
 \begin{array}{ccc}
 && b = (w^{(1)}_0)^2-1, ~c=(w^{(2)}_0)^2-1, \\
&& d=1- (w^{(1)}_0)^2 - (w^{(2)}_0)^2 + (w^{(1)}_0)^2 (w^{(2)}_{\varepsilon_1})^2.
 \end{array}
\right\} \eeq
If \eqref{wt-3-iso-new} holds, then $\mathscr W$ is a toral $3$-isometry. 
\end{enumerate} 
 \end{proposition}
 \begin{proof} 
 (i) The necessity part follows from Theorem~\ref{thm-char-m-iso}. To see the sufficiency part, assume that \eqref{wt-2-iso} holds. 
Note that for any $\beta \in \mathbb Z^2_+,$ there exists a positive scalar $m(\beta)$ such that
 \beq \label{W-beta}
 \mathscr W^{\beta} e_0= m(\beta)e_\beta.
 \eeq
 It now follows that for some real scalars $b_{\beta, \delta}$ 
 \beqn
 \|\mathscr W^\alpha e_\beta\|^2 = \frac{1}{m(\beta)^2} \|\mathscr W^{\alpha + \beta} e_0\|^2 \overset{\eqref{wt-2-iso}}=  \sum_{k=0}^{m-1} \sum_{\overset{\delta \in \mathbb Z^2_+}{|\delta|=k}} b_{\beta, \delta}\, (\alpha + \beta)_\delta, \quad \alpha, \beta \in \mathbb Z^2_+.
 \eeqn
 This combined with \eqref{fact-m-iso} yields
 \beqn
\sum_{\overset{\alpha \in \mathbb Z^2_+}{0 \Le \alpha \Le \gamma}} (-1)^{|\alpha|}\binom{\gamma}{\alpha}\|\mathscr W^{\alpha}e_\beta\|^2 = 0, \quad \beta \in \mathbb Z^2_+,~ \gamma \in \mathbb Z^2_+, ~|\gamma|=m.
\eeqn
Since $\{\mathscr W^\alpha e_\beta\}_{\beta \in \mathbb Z^2_+}$ is an orthogonal set for every $\alpha \in \mathbb Z^2_+,$ we conclude that 
\beqn
\sum_{\overset{\alpha \in \mathbb Z^2_+}{0 \Le \alpha \Le \gamma}} (-1)^{|\alpha|}\binom{\gamma}{\alpha}\|\mathscr W^{\alpha}h\|^2 = 0, \quad h \in H,~ \gamma \in \mathbb Z^2_+, ~|\gamma|=m.
\eeqn
This shows that $\mathscr W$ is a toral $m$-isometry. 

(ii) Assume that \eqref{wt-2-iso} holds. Similar to the verification of (i), this may be deduced from \eqref{W-beta} and Corollary~\ref{coro-char-m-iso}. 
 \end{proof}
 \begin{remark} 
 \label{sep-toral-ex}
 Note that $\mathscr W$ is a toral $2$-isometry if and only if there exist unique nonnegative numbers $b, c$ such that
 \beqn 
 \|\mathscr W^\alpha e_0\|^2  = 1 + b \alpha_1 + c \alpha_2, \quad \alpha=(\alpha_1, \alpha_2) \in \mathbb Z^2_+.
 \eeqn
Moreover, $b=(w^{(1)}_0)^2-1$ and $c=(w^{(2)}_0)^2-1.$
These observations were implicitly recorded in \cite[Remark~3]{AS1999}.
Thus, for any $\alpha \in \mathbb Z^2_+,$
 \beqn
 w^{(1)}_\alpha = \frac{\|\mathscr W^{\alpha + \varepsilon_1}e_0\|}{\|\mathscr W^{\alpha}e_0\|}=\sqrt{\frac{1 + ((w^{(1)}_0)^2-1) (\alpha_1+1) +  ((w^{(1)}_0)^2-1) \alpha_2}{1 + ((w^{(1)}_0)^2-1) \alpha_1 +  ((w^{(1)}_0)^2-1) \alpha_2}}.
 \eeqn
 Similarly, one can see that
 \beqn
 w^{(2)}_\alpha = \sqrt{\frac{1 + ((w^{(1)}_0)^2-1) \alpha_1 +  ((w^{(1)}_0)^2-1) (\alpha_2 + 1)}{1 + ((w^{(1)}_0)^2-1) \alpha_1 +  ((w^{(1)}_0)^2-1) \alpha_2}}, \quad \alpha \in \mathbb Z^2_+.
 \eeqn
This provides $2$-variable counterpart of \cite[Lemma~6.1(ii)]{JS2001}.

For a polynomial $p(\alpha)=1 + b \alpha_1 
+ c \alpha_2 + d \alpha_1 \alpha_2$ with nonnegative real numbers $b, c, d,$  consider the weighted $2$-shift $\mathscr W_p$ with weights  
\beq \label{wts-special}
w^{(j)}_\alpha = \sqrt{\frac{p(\alpha 
+ \varepsilon_j)}{p(\alpha)}}, \quad \alpha \in \mathbb 
Z^2_+, ~j=1, 2.
\eeq
By Proposition~\ref{m-iso-wt-char}, $\mathscr W_p$ is a separate $2$-isometry. If $d >0,$ then $\mathscr W_p$ is never a toral $2$-isometry (see Corollary~\ref{last-coro}).
 \hfill $\diamondsuit$
 \end{remark}
 
 Let $T=(T_1, \ldots, T_n)$ be a commuting $n$-tuple on $H.$ 
We say that $T$ is {\it jointly subnormal} if there exist a Hilbert space $K$ containing $H$ and a commuting $n$-tuple $N$ of normal operators $N_1, \ldots, N_n$ on $K$ such that
\beqn
T_j = {N_j}|_{H}, \quad j=1, \ldots, n.
\eeqn
Assume that $T^*_jT_j$ is invertible for every $j=1, \ldots, n.$ Following \cite{S2001, CC2012}, 
we refer to the $m$-tuple $T^{\mathfrak{t}}:=
(T^{\mathfrak{t}}_1, \ldots, T^{\mathfrak{t}}_n)$ as the \textit{operator tuple
torally Cauchy dual} to $T$, where $T^{\mathfrak{t}}_j :=
T_j(T^*_jT_j)^{-1},$ $j=1, \ldots, n.$ 
\begin{remark} 
\label{app-Richter}
Let $T=(T_1, \ldots, T_n)$ be a separate $2$-isometry on $H.$ By Richter's lemma (see \cite[Lemma~1]{R1988}), 
\beq \label{expansion}
T^*_jT_j \Ge I, \quad j=1, \ldots, n.
\eeq
It follows that the operator $n$-tuple $T^{\mathfrak{t}}$
torally Cauchy dual to $T$ exists and 
\beq \label{contraction}
(T^{\mathfrak{t}}_j)^* T^{\mathfrak{t}}_j \Le I, \quad j=1, \ldots, n. 
\eeq
Note that the operator $n$-tuple 
torally Cauchy dual to $T^{\mathfrak{t}}$ exists and it is equal to $T.$
\hfill $\diamondsuit$
\end{remark}

Let $\mathscr W : \{w^{(j)}_\alpha\}$ be a weighted $n$-shift such that $\mathscr W^*_j \mathscr W_j$ is invertible for every $j=1, \ldots, n.$ The 
operator tuple $\mathscr W^{\mathfrak{t}}$
torally Cauchy dual to the weighted $n$-shift $W$ satisfies 
\beq \label{dual-action} \mathscr W^{\mathfrak{t}}_j e_\alpha =
\frac{1}{w^{(j)}_\alpha} \, e_{\alpha + \varepsilon_j}, \quad j=1, \ldots, n. 
\eeq
Hence, by \eqref{commuting}, $\mathscr W^{\mathfrak{t}}$ is commuting$:$ 
\beqn 
\mathscr W^{\mathfrak{t}}_i \mathscr W^{\mathfrak{t}}_j= \mathscr W^{\mathfrak{t}}_j \mathscr W^{\mathfrak{t}}_i,\quad 1 \Le i \neq  j \Le n. 
\eeqn 
Moreover, by \eqref{dual-action}, 
 \beq \label{reciprocal}
 \|(\mathscr W^{\mathfrak{t}})^\alpha e_0\|^2  = \frac{1}{\|\mathscr W^{\alpha} e_0\|^2}, \quad \alpha  \in \mathbb Z^n_+.
 \eeq

We now present a complete solution to the Cauchy dual subnormality problem for the class of toral $3$-isometric weighted $2$-shifts, which are separate $2$-isometries. 
\begin{theorem} \label{CDSP-2-iso-wt}
Let $\mathscr W : \{w^{(j)}_\alpha\}$ be a  toral $3$-isometric weighted $2$-shift.
If $\mathscr W$ is a separate $2$-isometry, then the operator tuple $\mathscr W^{\mathfrak{t}}$
torally Cauchy dual to $\mathscr W$ exists, and the following statements are equivalent$:$
\begin{enumerate}
\item[$\mathrm{(i)}$] the operator tuple $\mathscr W^{\mathfrak{t}}$
 is jointly subnormal, 
\item[$\mathrm{(ii)}$] either $1- (w^{(1)}_0)^2 - (w^{(2)}_0)^2 + (w^{(1)}_0)^2 (w^{(2)}_{\varepsilon_1})^2 = 0$ or 
$w^{(2)}_{\varepsilon_1} \Le w^{(2)}_0,$
 \item[$\mathrm{(iii)}$] either $\mathscr W$ is a toral $2$-isometry or $w^{(2)}_{\varepsilon_1} \Le w^{(2)}_0.$
 \end{enumerate}
\end{theorem}
\begin{proof}
Assume that $\mathscr W$ is a separate $2$-isometry. 
By Remark~\ref{app-Richter}, the operator tuple $\mathscr W^{\mathfrak{t}}$
torally Cauchy dual to $\mathscr W$ exists. 
Also, by Proposition~\ref{m-iso-wt-char}(ii) and \eqref{reciprocal},
\beq \label{last-eq}
 \|(\mathscr W^{\mathfrak{t}})^\alpha e_0\|^2  = 
\frac{1}{1 + \alpha_1 b + 
 \alpha_2 (c + \alpha_1 d)}, \quad \alpha=(\alpha_1, \alpha_2) \in \mathbb Z^2_+,
 \eeq
 where $b, c, d$ are given by  \eqref{value-d}. 
Note that by \eqref{expansion}, $b \Ge 0$ and $c \Ge 0.$
 One may now apply \cite[Theorem 4.4]{Athavale1987} together with \eqref{W-beta} and \eqref{contraction} to see that 
  \begin{align} \label{subnor-if-r}
   \begin{minipage}{67ex}
\text{$\mathscr W^{\mathfrak{t}}$ is jointly subnormal if and only if $\{\|(\mathscr W^{\mathfrak{t}})^\alpha e_0\|^2\}_{\alpha \in \mathbb Z^2_+}$ is a joint} 
\\
\text{completely monotone net} 
\end{minipage}
 \end{align}
(see also the discussion prior to \cite[Eqn~(E)]{Athavale2001}). On the other hand, by  Theorem~\ref{deg-1-1} and \eqref{last-eq}, 
 \begin{align} \label{subnor-CDSP}
   \begin{minipage}{67ex}
\text{$\{\|(\mathscr W^{\mathfrak{t}})^\alpha e_0\|^2\}_{\alpha \in \mathbb Z^2_+}$ is joint completely monotone $\Longleftrightarrow$ $d \Le bc.$}  
 \end{minipage}
 \end{align} 
 To get the equivalence of (i) and (ii), note that if $d \neq 0,$ then by \eqref{value-d}, $d \Le bc$ if and only if 
\beqn
(w^{(1)}_0)^2((w^{(2)}_{\varepsilon_1})^2-1) \Le c(b+1) =((w^{(2)}_0)^2-1)(w^{(1)}_0)^2, 
\eeqn
which is equivalent to 
\beqn
(w^{(1)}_0)^2((w^{(2)}_0)^2 - (w^{(2)}_{\varepsilon_1})^2) \Ge 0 ~\Longleftrightarrow w^{(2)}_{\varepsilon_1} \Le w^{(2)}_0.
\eeqn
The equivalence of (i) and (ii) now follows from \eqref{value-d}, \eqref{subnor-if-r} and \eqref{subnor-CDSP}.
Finally, since a separate $2$-isometry $\mathscr W$ is a toral $2$-isometry if and only if  $d=0$ (see Proposition~\ref{m-iso-wt-char}(ii)), the equivalence of (ii) and (iii) is immediate. 
\end{proof}
\begin{remark} 
By \eqref{commuting}, $w^{(2)}_{\varepsilon_1} \Le w^{(2)}_0$ if and only if $w^{(1)}_{\varepsilon_2} \Le w^{(1)}_0.$ 
 \hfill $\diamondsuit$
\end{remark}

The first part of the following corollary may also be deduced from \cite[Proposition~6]{AS1999}.
\begin{corollary} \label{yes-C-dual}
Under the hypotheses of Theorem~\ref{CDSP-2-iso-wt}, the following statements are valid$:$
\begin{enumerate}
\item[$\mathrm{(i)}$] if $\mathscr W$ is a toral $2$-isometry, $\mathscr W^{\mathfrak{t}}$
 is jointly subnormal, 
\item[$\mathrm{(ii)}$] if $\mathscr W$ is not a toral $2$-isometry, then $\mathscr W^{\mathfrak{t}}$
 is jointly subnormal if and only if 
$w^{(2)}_{\varepsilon_1} \Le w^{(2)}_0.$
\end{enumerate}
\end{corollary}
 
 We conclude the section with an example of a toral $3$-isometry for which the operator tuple torally Cauchy dual is not jointly subnormal. 
 \begin{example} \label{not-C-dual} For real numbers $a, b, c, d >0,$
consider the polynomial $p: \mathbb R^2_+ \rar (0, \infty)$ given by $$p(\alpha)=1 + \alpha_1 b + 
 \alpha_2 (c + \alpha_1 d).$$  
 Let $\mathscr W_p$ be the weighted $2$-shift with weights given by \eqref{wts-special}. By Proposition~\ref{m-iso-wt-char}, $\mathscr W_p$ is a joint $3$-isometry, which is also a separate $2$-isometry. 
Further,  
by the discussion following \eqref{subnor-CDSP}, $w^{(2)}_{\varepsilon_1} \Le w^{(2)}_0$ if and only if $d \Le bc.$
It is now clear from Theorem~\ref{CDSP-2-iso-wt} that there exist joint $3$-isometries $\mathscr W$ for which $\mathscr W^{\mathfrak{t}}$ is not jointly subnormal (for example, let $b=1,$ $c=2$ and $d=3$). \eof 
\end{example}

\section{Joint complete monotonicity and a coefficient-matrix}

A solution to the 
Cauchy dual subnormality problem for toral $3$-isometric weighted $2$-shifts requires characterization of polynomials $p : \mathbb R^2_+ \rar (0, \infty)$ of bi-degree at most $(2, 2)$ for which $\big\{\frac{1}{p(m, n)}\big\}_{m,n \in \mathbb{Z}_+}$ is joint completely monotone. We have already seen one special instance of this problem (see Theorem~\ref{CDSP-2-iso-wt}). 
In the remaining part of this paper, we briefly discuss role of  the coefficient-matrix in the classification of the polynomials $p : \mathbb R^2_+ \rar (0, \infty)$ for which $\frac{1}{p}$ is joint completely monotone. 
For a positive integer $N$ and $p_N(x, y) = \sum_{m, n=0}^N a_{m, n} x^m y^n,$ consider the {\it coefficient-matrix} $A_{p_N} = (a_{m, n})_{m, n =0}^N$ associated with $p_N.$ The equivalence of (i) and (ii) of Theorem~\ref{deg-1-1} can be rephrased as follows$:$ 
\begin{theorem} \label{matrix-i} The function
$\frac{1}{p_1}$ is joint completely monotone if and only if 
$\det A_{p_1} \Le 0.$
\end{theorem}

Consider a polynomial $p_2$ of the form $(x+b_1)(x+b_2) + (x+a_1)(x+a_2)y,$ where $0 < a_1 \Le a_2$ and $0 < b_1 \Le b_2.$ Note that 
\beqn
A_{p_2} = \begin{pmatrix}
b_1b_2 & a_1a_2  & 0\\
b_1 + b_2 & a_1 + a_2 & 0 \\
1 & 1 & 0
\end{pmatrix}.
\eeqn
Clearly, $\det A_{p_2} =0.$ If $B_{p_2}$ denotes the minor of $A_{p_2}$ obtained after excluding the third row and the third column, then  
$$\det B_{p_2} = b_1b_2(a_1+a_2) - a_1a_2(b_1+b_2).$$  
Note that the condition $\det B_{p_2} \Le 0$ is precisely the condition \eqref{necessary-c-2} with $k=2.$  Hence, by Lemma~\ref{lem-necessary}, $\det B_{p_2} \Le 0$ is a necessary condition for $\frac{1}{p_2}$ to be joint completely monotone.   
On the other hand, by Remark~\ref{following}, $\frac{1}{p_2}$ is joint completely monotone provided \eqref{assum-2-deg} holds. Thus the condition 
\eqref{assum-2-deg} implies that $\det B_{p_2} \Le 0.$ 
Summarizing the discussion above, we have
\beq \label{matrix-s-n}
\eqref{assum-2-deg} \Longrightarrow \text{$\frac{1}{p_2}$ is joint completely monotone} \Longrightarrow \det B_{p_2} \Le 0.
\eeq
Interestingly, $\det B_{p_2} \Le 0$ neither ensures \eqref{assum-2-deg} nor the joint  complete monotonicity of $\frac{1}{p_2}$ (for the first assertion, consider $a_1 = 6,$ $a_2=9,$ $b_1=1$ and $b_2=5,$ and for the second one, see Example~\ref{exam-1.3-i}(i)). Needless to say, the problem of finding same set of sufficient and necessary conditions for $\frac{1}{p_2}$ to be joint completely monotone remains unresolved. 

\begin{acknowledgments}
We convey our sincere thanks to Md. Ramiz Reza and Paramita Pramanick for several useful comments on the initial draft of this paper.
\end{acknowledgments}

{}

\end{document}